\documentclass[12pt,reqno]{amsart}

\usepackage{amsbsy}
\usepackage{amscd}
\usepackage{amsfonts}
\usepackage{amsmath}
\usepackage{amssymb}
\usepackage{amsthm}
\usepackage{amsxtra}

\usepackage[utf8]{inputenc}
\usepackage[margin=1in]{geometry}

\usepackage{comment}
\usepackage{datetime}
\usepackage{tikz-cd}
\usepackage{tikz}
\usepackage{todonotes}
\usepackage{verbatim}

\theoremstyle{definition}

\newtheorem{thm}{Theorem}[section]
\newtheorem*{thm*}{Theorem}

\newtheorem*{conj*}{Conjecture}

\newtheorem*{conv*}{Convention}
\newtheorem{cor}[thm]{Corollary}
\newtheorem*{cor*}{Corollary}

\newtheorem*{defn*}{Definition}

\newtheorem*{exa*}{Example}
\newtheorem{exc}[thm]{Exercise}
\newtheorem*{exc*}{Exercise}

\newtheorem*{fact*}{Fact}
\newtheorem{lem}[thm]{Lemma}
\newtheorem*{lem*}{Lemma}

\newtheorem*{prob*}{Problem}

\newtheorem*{prop*}{Proposition}

\newtheorem*{ques*}{Question}

\newtheorem*{rmk*}{Remark}

\newcommand{\C}{\mathbb{C}}

\renewcommand{\H}{\mathbb{H}}

\newcommand{\Q}{\mathbb{Q}}

\newcommand{\R}{\mathbb{R}}
\newcommand{\Z}{\mathbb{Z}}

\newcommand{\cA}{\mathcal{A}}
\newcommand{\cC}{\mathcal{C}}
\newcommand{\cE}{\mathcal{E}}

\newcommand{\cM}{\mathcal{M}}

\newcommand{\cS}{\mathcal{S}}
\newcommand{\cU}{\mathcal{U}}

\newcommand{\al}{\alpha}
\newcommand{\bet}{\beta}
\newcommand{\Gam}{\Gamma}
\newcommand{\gam}{\gamma}

\newcommand{\del}{\delta}
\newcommand{\eps}{\varepsilon}

\newcommand{\kap}{\kappa}
\newcommand{\Lam}{\Lambda}
\newcommand{\lam}{\lambda}

\newcommand{\sig}{\sigma}

\newcommand{\Om}{\Omega}
\newcommand{\om}{\omega}

\newcommand{\ra}{\rightarrow}

\newcommand{\ol}{\overline}
\newcommand{\pr}{\prime}

\newcommand{\wt}{\widetilde}

\newcommand{\mat}[4]{\left(\begin{array}{cc} #1 & #2 \\ #3 & #4 \end{array}\right)}

\newcommand{\hmod}{\hspace{-4pt}\mod} 

\newcommand{\sm}{\setminus}

\DeclareMathOperator{\GL}{GL}
\DeclareMathOperator{\SL}{SL}

\DeclareMathOperator{\ind}{ind}

\DeclareMathOperator*{\Rel}{Rel}

\DeclareMathOperator{\supp}{supp}

\newcommand{\be}{\begin{equation*}}
\newcommand{\ee}{\end{equation*}}
\newcommand{\bex}{\begin{exc}}
\newcommand{\eex}{\end{exc}}
\newcommand{\bpf}{\begin{proof}}
\newcommand{\epf}{\end{proof}}

\settimeformat{hhmmsstime}

\title{Dense real Rel flow orbits and absolute period leaves}
\author[Karl Winsor]{Karl Winsor \\ \\ \monthname[\the\month] \the\day, \the\year}

\begin{document}

\maketitle

\begin{abstract}
We show the existence of a dense orbit for real Rel flows on the area-$1$ locus of every connected component of every stratum of holomorphic $1$-forms with at least $2$ distinct zeros. For this purpose, we establish a general density criterion for $\SL(2,\R)$-orbit closures, based on finding an orbit of a real Rel flow whose closure contains a horocycle. This criterion can be verified using explicit constructions of holomorphic $1$-forms with a periodic horizontal foliation. Our constructions also provide explicit examples of dense leaves of the absolute period foliation and many subfoliations of these loci.
\end{abstract}


\section{Introduction} \label{sec:intro}

This paper contributes to the study of the dynamics of flows and foliations on spaces of holomorphic $1$-forms. Fix $g \geq 2$ and a partition $\kap = (k_1,\dots,k_n)$ of $2g-2$. The stratum $\Om\cM_g(\kap)$ is the moduli space of nonzero holomorphic $1$-forms on closed Riemann surfaces of genus $g$ with zero orders given by $\kap$. We denote elements of $\Om\cM_g(\kap)$ by pairs $(X,\om)$. The holomorphic $1$-form $\om$ provides geometric structures on $X$, including a flat metric and a foliation by horizontal lines, with singularities along the zero set $Z(\om)$.

When $n > 1$, it is possible to deform $(X,\om)$ by moving the zeros of $\om$ relative to each other without changing the integrals of $\om$ along closed loops in $X$. Let $\C^{Z(\om)}/\C(1,\dots,1)$ be the vector space of functions $Z(\om) \ra \C$ modulo constants. A nonzero $v \in \C^{Z(\om)}/\C(1,\dots,1)$ determines a family of holomorphic $1$-forms $(X_t,\om_t) \in \Om\cM_g(\kap)$ such that for any $c \in H_1(X;\Z)$ and zeros $Z,Z^\pr$ of $\om$, we have
\be
\frac{d}{dt} \int_c \om_t = 0, \quad \frac{d}{dt} \int_Z^{Z^\pr} \om_t = v_{Z^\pr} - v_Z .
\ee
The holomorphic $1$-forms $(X_t,\om_t)$ are obtained from $(X,\om)$ by slitting and regluing along certain segments emanating from the zeros of $\om$. After passing to a finite cover $\wt{\Om}\cM_g(\kap) \ra \Om\cM_g(\kap)$ obtained by labelling the zeros $Z_1,\dots,Z_n$, these deformations determine a partially defined flow. Define
\be
\C^n_0 = \C^n/\C(1,\dots,1), \quad \R^n_0 = \R^n/\R(1,\dots,1) .
\ee
When $v \in \R^n_0$, this flow is called a {\em real Rel flow}. Orbits of real Rel flows are not always well-defined for all time, since distinct zeros may collide in finite time, but it is known that this is the only obstruction \cite{BSW:horocycle}, \cite{MW:cohomology}. We denote the flow orbit through $(X,\om)$ by $\Rel_{\R v}(X,\om)$ when it is well-defined for all time. \\

\paragraph{\bf Main results.} Our two main results address the topological dynamics of real Rel flows and the foliations obtained by ``complexifying'' these flows. We state these results in Theorems \ref{thm:denserealRel} and \ref{thm:denseV} below, and we outline our methods of proof.

Since the area of a holomorphic $1$-form is constant along a real Rel flow orbit, let $\wt{\Om}_1\cM_g(\kap)$ be the area-$1$ locus in $\wt{\Om}\cM_g(\kap)$. A basic question in topological dynamics is the existence of a dense orbit. Our first main result completely answers this question for real Rel flows on strata.

\begin{thm} \label{thm:denserealRel}
Suppose $n > 1$, fix a nonzero $v \in \R^n_0$, and let $\wt{\cC}_1$ be a connected component of $\wt{\Om}_1\cM_g(\kap)$. There exists $(X,\om) \in \wt{\cC}_1$ such that $\Rel_{\R v}(X,\om)$ is dense in $\wt{\cC}_1$.
\end{thm}

\noindent
Previously, Theorem \ref{thm:denserealRel} was unknown even in the case $\kap = (1,1)$.

When $n > 1$, the stratum $\Om\cM_g(\kap)$ also admits a holomorphic foliation $\cA(\kap)$ called the {\em absolute period foliation}. Two holomorphic $1$-forms lie on the same leaf of $\cA(\kap)$ if and only if there is a path between them along which the integrals along closed loops are constant. Leaves of $\cA(\kap)$ have a natural locally Euclidean metric, recorded by integrals along paths between distinct zeros. Orbits of real Rel flows project to geodesics on leaves of $\cA(\kap)$, so our density result has consequences for the absolute period foliation.

\begin{cor} \label{cor:denseleaf}
Suppose $n > 1$, and let $\cC_1$ be a connected component of $\Om_1\cM_g(\kap)$. There exists a leaf of $\cA(\kap)$ that is dense in $\cC_1$.
\end{cor}

In \cite{Win:ergodic}, a stronger density result is proven for leaves of $\cA(\kap)$ in the case where $n > 1$ and $\Om\cM_g(\kap)$ is connected. The above results are not explicit. However, our proofs will involve explicit constructions and will provide explicit examples of dense leaves of $\cA(\kap)$ and of many foliations whose leaves are contained in leaves of $\cA(\kap)$. Our second main result will provide a stronger and explicit version of Corollary \ref{cor:denseleaf}. \\

\paragraph{\bf Real Rel and horocycles.} Our study of real Rel flows is based on a study of their centralizers and normalizers. In our setting, we are most interested in the action of $\SL(2,\R)$ and its subgroups on $\wt{\Om}\cM_g(\kap)$. The horocycle flow and the geodesic flow are defined by the actions of
\be
u_s = \mat{1}{s}{0}{1}, \quad g_t = \mat{e^t}{0}{0}{e^{-t}},
\ee
respectively. Real Rel flows commute with the horocycle flow, and both are normalized by the geodesic flow.

One of the novelties in our approach to Theorem \ref{thm:denserealRel} is an application of a recent result in \cite{For:pushforward} to the study of real Rel flows. Suppose that the closure of $\Rel_{\R v}(X,\om)$ contains the horocycle through $(X,\om)$. Since the geodesic flow normalizes both real Rel flows and the horocycle flow, this property is invariant under the geodesic flow. By Corollary 1.3 in \cite{For:pushforward}, pushforwards of horocycle arcs under the geodesic flow equidistribute in the associated $\SL(2,\R)$-orbit closure, outside a sequence of times of zero upper density. By applying $g_{t_n}$ for an appropriate sequence of times $t_n$, we obtain real Rel flow orbits that become more and more dense in an $\SL(2,\R)$-orbit closure. A short argument using the Baire category theorem then yields the existence of a real Rel flow orbit whose closure contains an $\SL(2,\R)$-orbit closure. Thus, we obtain a simple and general criterion for the existence of a dense real Rel flow orbit in an $\SL(2,\R)$-orbit closure. We present this argument in Section \ref{sec:densecriterion}.

\begin{thm} \label{thm:densecriterion1}
Suppose $n > 1$, and fix a nonzero $v \in \R^n_0$. Fix $(X,\om) \in \wt{\Om}\cM_g(\kap)$, and let $\cM \subset \wt{\Om}\cM_g(\kap)$ be the $\SL(2,\R)$-orbit closure of $(X,\om)$. If $\Rel_{\R v} (X,\om)$ is contained in $\cM$ and the closure of $\Rel_{\R v} (X,\om)$ contains the horocycle through $(X,\om)$, then there exists $(Y,\eta) \in \cM$ such that $\Rel_{\R v}(Y,\eta)$ is dense in $\cM$.
\end{thm}

Our second observation is that holomorphic $1$-forms $(X,\om)$ with a periodic horizontal foliation can provide explicit examples for which the closure of $\Rel_{\R v}(X,\om)$ contains the horocycle through $(X,\om)$. In this case, $X$ is covered by finitely many cylinders of horizontal closed geodesics along with finitely many horizontal geodesic segments between zeros. The horocycle through $(X,\om)$ is given by twisting each horizontal cylinder. The heights and circumferences of the horizontal cylinders, and the lengths of horizontal geodesic segments, are constant along the horocycle. The heights and circumferences of horizontal cylinders are also constant along a real Rel flow orbit, but the lengths of certain horizontal geodesic segments between zeros may change. However, if $v_i - v_j = 0$ whenever $Z_i$ and $Z_j$ are the endpoints of a horizontal geodesic segment, then $\Rel_{\R v}(X,\om)$ can also be described by twisting the horizontal cylinders. In this setting, both the horocycle flow and the real Rel flow can essentially be thought of as linear flows on a compact torus parametrized by twist parameters for the horizontal cylinders. This was previously described in \cite{HW:Rel} in the case of holomorphic $1$-forms with exactly two zeros. In Section \ref{sec:twist}, we show that if every horizontal cylinder is twisted along $\Rel_{\R v}(X,\om)$, and if distinct horizontal cylinders do not contain homologous closed geodesics, then it is possible to perturb the horizontal cylinder circumferences so that $\Rel_{\R v}(X,\om)$ is dense in this twist torus and so that the $\SL(2,\R)$-orbit of $(X,\om)$ is dense in the area-$1$ locus of its stratum component. For the latter density property, we rely on the explicit density criterion for $\GL^+(2,\R)$-orbits in strata from \cite{Wri:field}. Theorem \ref{thm:denserealRel} is thus reduced to the following construction which is carried out in Section \ref{sec:construction}.

\begin{thm} \label{thm:constr1}
Suppose $n > 1$, fix a nonzero $v \in \R^n_0$, and let $\wt{\cC}$ be a connected component of $\wt{\Om}\cM_g(\kap)$. There exists $(X,\om) \in \wt{\cC}$ with a periodic horizontal foliation satisfying the following properties.
\begin{enumerate}
    \item If $Z_i$ and $Z_j$ are the endpoints of a horizontal geodesic segment, then $v_i - v_j = 0$.
    \item If $Z_i$ and $Z_j$ are zeros in the top and bottom boundaries, respectively, of the same horizontal cylinder, then $v_i - v_j \neq 0$.
    \item Distinct horizontal cylinders do not contain homologous closed geodesics.
\end{enumerate}
\end{thm}

The constructions in the proof of Theorem \ref{thm:constr1} are rather involved and occupy the bulk of this paper. Briefly, we first use an intricate connected sum of tori to establish the case where every $k_j \in \kap$ satisfies $k_j \leq g - 1$ and the components of $v$ are distinct, and we then deal with the remaining cases using well-known surgeries.

Our second main result is an explicit density result for ``complexified'' real Rel flow orbits. A real vector subspace $V$ of $\C^n_0$ determines a foliation $\cA_V(\kap)$ of $\wt{\Om}\cM_g(\kap)$ whose leaves project into leaves of $\cA(\kap)$. Our second main result provides explicit dense leaves of many subfoliations of $\cA(\kap)$ with complex $1$-dimensional leaves in all stratum components with multiple zeros.

\begin{thm} \label{thm:denseV}
Suppose $n > 1$, fix a nonzero $v \in \R_0^n$ with distinct components, and let $V = \C v$. Let $\wt{\cC}_1$ be a connected component of $\wt{\Om}_1\cM_g(\kap)$. There is an explicit $(X,\om) \in \wt{\cC}_1$ with periodic horizontal and vertical foliations, such that the leaf of $\cA_V(\kap)$ through $(X,\om)$ is dense in $\wt{\cC}_1$.
\end{thm}

A novelty in our approach to Theorem \ref{thm:denseV} is that it is based on analyzing periodic foliations in multiple directions simultaneously. Much of the proof of Theorem \ref{thm:denseV} is contained in the proof of Theorem \ref{thm:constr1}. In most cases, the holomorphic $1$-forms we construct here have a periodic vertical foliation with exactly analogous properties. In particular, we end up constructing explicit holomorphic $1$-forms $(X,\om)$ such that the closure of $\Rel_{\R v}(X,\om)$ contains the horocycle through $(X,\om)$, and such that the closure of $\Rel_{i\R v}(X,\om)$ contains the opposite horocycle through $(X,\om)$.

We expect that the hypothesis that $v$ has distinct components can be removed. However, with our approach, it seems that doing so would greatly increase the length of this paper, and we decided not to pursue this. Part of our motivation for proving Theorem \ref{thm:denseV} is as a complement to Theorem 1.2 in \cite{Win:ergodic}, which produces explicit full measure sets of dense leaves of $\cA(\kap)$ in $\Om_1\cM_g(\kap)$. The holomorphic $1$-forms in these dense leaves of $\cA(\kap)$ are in some sense as far as possible from having a periodic foliation in any direction. We hope that the methods in these two papers can be combined to make progress toward classifying the closures of leaves of $\cA(\kap)$ and $\cA_V(\kap)$. Lastly, we remark that the criterion in Theorem \ref{thm:densecriterion1} can be readily verified for other $\SL(2,\R)$-orbit closures, and we briefly discuss this in Section \ref{sec:construction} as well. \\

\paragraph{\bf Notes and references.} In \cite{BSW:horocycle} and \cite{MW:cohomology}, it is shown that the only obstruction to a real Rel flow orbit being well-defined for all time comes from horizontal saddle connections. Related completeness results for leaves of $\cA(\kap)$ are proven in \cite{McM:navigating} and \cite{McM:isoperiodic}. In the terminology of \cite{McM:navigating}, Theorem \ref{thm:denserealRel} implies the existence of dense relative period geodesics in the area-$1$ locus of every connected component of $\Om\cM_g(k_1,\dots,k_n)$ when $n > 1$.

The study of real Rel flows and the horocycle flow on strata are often intertwined. In \cite{BSW:horocycle}, real Rel flows play an important role in the classification of orbit closures and invariant measures for the horocycle flow on the eigenform loci in $\Om\cM_2(1,1)$. It would be interesting to see if closures of real Rel flow orbits in these eigenform loci can also be classified. In \cite{MW:horocycle}, it is shown that horocycles in strata do not diverge, and that almost every holomorphic $1$-form in a horocycle has a uniquely ergodic vertical foliation. Similar results for families of interval exchange transformations arising from horocycles are proven in \cite{MW:cohomology}. In contrast, \cite{HW:Rel} gives an example of a real Rel flow orbit that is well-defined for all time but diverges in its stratum. Moreover, in \cite{HW:Rel}, it is shown that the Arnoux-Yoccoz surface in genus $3$ gives an example for which there is a unique holomorphic $1$-form in a real Rel flow orbit with a uniquely ergodic vertical foliation, while all others have a periodic vertical foliation. Other exotic behaviors of vertical foliations along relative period geodesics are exhibited in \cite{McM:isoperiodic} and \cite{McM:cascades}. In \cite{CSW:tremor}, it is shown that closures of horocycles in $\Om\cM_2(1,1)$ can have non-integer Hausdorff dimension. In light of \cite{CSW:tremor}, it would be interesting to see if closures of real Rel flow orbits in $\Om\cM_2(1,1)$ can also have non-integer Hausdorff dimension.

In \cite{Ygo:dense}, a related criterion to our Theorem \ref{thm:densecriterion1} is given for the density of leaves of the absolute period foliation in the special case of affine invariant manifolds of rank $1$ (defined in \cite{Wri:cylinder}). The criterion in \cite{Ygo:dense} is based on finding horizontally periodic surfaces where the horizontal cylinders can be twisted while staying in the absolute period leaf and accumulating on a horocycle. Horocycle invariance of the leaf closure is promoted to $\SL(2,\R)$-invariance, and thus density due to the rank $1$ condition, using the existence of holomorphic $1$-forms with a hyperbolic Veech group element in every absolute period leaf in rank $1$. A key difference in our criterion is the use of Corollary 1.3 in \cite{For:pushforward}, which is needed to obtain a density result for real Rel flow orbits (as opposed to the full absolute period foliation) as well as to obtain a general result for $\SL(2,\R)$-orbit closures. Additionally, hyperbolic Veech group elements do not play a role in our arguments. Jon Chaika and Barak Weiss have informed us of work in progress in which they prove that real Rel flows are ergodic on connected components of $\wt{\Om}_1\cM_g(\kap)$ when $n > 1$, conditional on an extension of the results of \cite{EM:stationary} to products of strata. We remark that their ergodicity result does not imply Theorem \ref{thm:denseV}, since holomorphic $1$-forms in a stratum with a periodic foliation form a measure zero subset.

The dynamics of $\cA(\kap)$ on $\Om_1\cM_g(\kap)$ have been more extensively studied, for instance, in \cite{CDF:transfer}, \cite{Ham:ergodicity}, \cite{HW:Rel}, \cite{McM:isoperiodic}, \cite{Win:ergodic}, \cite{Ygo:dense}. Ergodicity of $\cA(1,\dots,1)$ on $\Om_1\cM_g(1,\dots,1)$ is proven in \cite{McM:isoperiodic} for $g = 2,3$, and for all $g \geq 2$ in \cite{CDF:transfer}, \cite{Ham:ergodicity}. Moreover, a classification of closures of leaves of $\cA(1,\dots,1)$ is given in \cite{CDF:transfer}. In \cite{HW:Rel} and \cite{Ygo:dense}, examples of dense leaves are given in a certain connected component of $\Om_1\cM_g(g-1,g-1)$ and in $\Om_1\cM_3(2,1,1)$. In \cite{Win:ergodic}, ergodicity of $\cA(\kap)$ on the area-$1$ locus of connected strata with $n > 1$ and the non-hyperelliptic connected component of $\Om\cM_g(g-1,g-1)$ for $g$ even is proven. Moreover, in \cite{Win:ergodic}, explicit full measure sets of dense leaves are given in these loci. \\

\paragraph{\bf Acknowledgements.} The author thanks Curt McMullen for encouragement and comments on earlier versions of this work. This material is based upon work supported by the National Science Foundation Graduate Research Fellowship Program under grant DGE-1144152.


\section{Background} \label{sec:background}

We review the flat geometry of holomorphic $1$-forms, strata of holomorphic $1$-forms, the $\GL^+(2,\R)$-action on strata, and the absolute period foliation of a stratum. We refer to \cite{BSW:horocycle}, \cite{McM:navigating}, and \cite{Zor:survey} for further background. \\

\paragraph{\bf Flat geometry.} Let $X$ be a closed Riemann surface of genus $g \geq 2$, and let $\om$ be a nonzero holomorphic $1$-form on $X$. The zero set $Z(\om)$ is finite, and the orders of the zeros form a partition $\kap = (k_1,\dots,k_n)$ of $2g-2$. For each $x \in X \sm Z(\om)$, there is a simply connected open neighborhood $U$ of $x$ and an injective holomorphic map $\phi : U \ra \C$ given by $\phi(z) = \int_x^z \om$ and satisfying $\om = \phi^\ast (dz)$. These maps provide an atlas of charts on $X \sm Z(\om)$ whose transition maps are translations. Translation-invariant structures on $\C$ can be pulled back to structures on $X$ with singularities at the zeros of $\om$. In particular, associated to the pair $(X,\om)$ is a flat metric with a cone point of angle $2\pi(k+1)$ at a zero of order $k$. Additionally, the foliations of $\C$ by horizontal and vertical lines determine a {\em horizontal foliation} and a {\em vertical foliation} of $(X,\om)$ each with $2k+2$ singular leaves meeting at a zero of order $k$.

A {\em saddle connection} on $(X,\om)$ is an oriented geodesic segment with endpoints in $Z(\om)$ and otherwise disjoint from $Z(\om)$. Any closed geodesic in $X \sm Z(\om)$ is contained in a maximal open {\em cylinder} given by a union of freely homotopic parallel closed geodesics. The boundary of a cylinder is a finite union of parallel saddle connections. Any cylinder $C$ is isometric to a unique Euclidean cylinder of the form $\R / w \Z \times (0,h)$ with $h,w > 0$, and $h$ and $w$ are the {\em height} and {\em circumference} of $C$, respectively. A {\em horizontal cylinder} is a cylinder containing a closed geodesic $\al$ such that $\int_\al \om \in \R_{>0}$. The orientation of $\al$ determines a {\em top boundary} of $C$ and a {\em bottom boundary} of $C$, which are not necessarily disjoint. Given a saddle connection $\gam$ which crosses a horizontal cylinder $C$ from bottom to top, the {\em twist parameter} of $C$ with respect to $\gam$ is defined by ${\rm Re} \int_\gam \om \in \R / w\Z$. We similarly define a {\em vertical cylinder} with $i\R_{>0}$ in place of $\R_{>0}$, as well as the {\em left boundary}, the {\em right boundary}, and the {\em twist parameter} of a vertical cylinder with respect to a saddle connection crossing the cylinder from right to left. Two cylinders are {\em homologous} if they contain closed geodesics that represent the same element of $H_1(X;\Z)$. The horizontal foliation of $(X,\om)$ is {\em periodic} if every leaf is compact. In this case, $(X,\om)$ is a union of finitely many disjoint horizontal cylinders and finitely many horizontal saddle connections. Similarly for the vertical foliation of $(X,\om)$. \\

\paragraph{\bf Strata.} The moduli space of holomorphic $1$-forms $\Om\cM_g$ classifies pairs $(X,\om)$ as above. The space $\Om\cM_g$ is a union of {\em strata} $\Om\cM_g(\kap)$ indexed by partitions $\kap = (k_1,\dots,k_n)$ of $2g-2$. A holomorphic $1$-form is in $\Om\cM_g(\kap)$ if and only if it has exactly $n$ distinct zeros of orders $k_1,\dots,k_n$. The bundle of relative homology groups $H_1(X,Z(\om);\Z)$ is locally trivial over $\Om\cM_g(\kap)$. Given $(X_0,\om_0) \in \Om\cM_g(\kap)$, we can define {\em period coordinates} on a small neighborhood of $(X_0,\om_0)$ by $(X,\om) \mapsto [\om] \in H^1(X_0,Z(\om_0);\C)$. By choosing a basis $\gam_1,\dots,\gam_{2g+n-1}$ for $H_1(X_0,Z(\om_0);\Z)$, we obtain a map
\be
(X,\om) \mapsto \left(\int_{\gam_1} \om,\dots,\int_{\gam_{2g+n-1}} \om\right) \in \C^{2g+n-1}
\ee
and the components $\int_{\gam_j} \om$ are called the {\em period coordinates of $(X,\om)$}. Period coordinates give $\Om\cM_g(\kap)$ the structure of a complex orbifold of dimension $2g + n - 1$. The {\em area} of $(X,\om)$ is given by $\frac{i}{2}\int_X \om \wedge \ol{\om}$, and the area-$1$ locus of $\Om\cM_g(\kap)$ is denoted by $\Om_1\cM_g(\kap)$.

We recall the classification of connected components of strata from \cite{KZ:components}. Suppose $k_1,\dots,k_n$ are even. Let $\gam \subset X \sm Z(\om)$ be a smooth oriented closed loop. Using the translation structure on $X \sm Z(\om)$, the {\em index} $\ind(\gam)$ is defined as $1/2\pi$ times the total change in angle along the loop $\gam$. In other words, $\ind(\gam)$ is the degree of the Gauss map $\gam \ra S^1$. Let $\{\al_j,\bet_j\}_{j=1}^g$ be a collection of smooth oriented closed loops in $X \sm Z(\om)$ that represents a symplectic basis for $H_1(X;\Z)$ with respect to the algebraic intersection form. The {\em parity of the spin structure} $\phi(\om)$ is defined by
\begin{equation} \label{eq:spin}
\phi(\om) = \sum_{j=1}^g (\ind(\al_j) + 1)(\ind(\bet_j) + 1) \hmod 2
\end{equation}
and is independent of the choice of symplectic basis of $H_1(X;\Z)$ and the choice of representative loops. Moreover, $\phi(\om)$ is an invariant of the connected component of $(X,\om)$ in $\Om\cM_g(\kap)$. A connected component is {\em even} or {\em odd} if it consists of holomorphic $1$-forms $(X,\om)$ with $\phi(\om) = 0$ or $\phi(\om) = 1$, respectively. A connected component is {\em hyperelliptic} if it consists of holomorphic $1$-forms on hyperelliptic Riemann surfaces with a unique zero, or if it consists of holomorphic $1$-forms on hyperelliptic Riemann surfaces with exactly two zeros that have equal order and are exchanged by the hyperelliptic involution. A connected component that is not hyperelliptic is {\em nonhyperelliptic}.

\begin{thm} \label{thm:KZ} (\cite{KZ:components}, Theorems 1-2, Corollary 5) For $g \geq 4$, the connected components of $\Om\cM_g(\kap)$ are given as follows.
\begin{enumerate}
    \item If $\kap = (2g-2)$ or $(g-1,g-1)$, then $\Om\cM_g(\kap)$ has a unique hyperelliptic component.
    \item If all $k_j$ are even, then $\Om\cM_g(\kap)$ has exactly two nonhyperelliptic components: one even component and one odd component.
    \item If some $k_j$ is odd, then $\Om\cM_g(\kap)$ has a unique nonhyperelliptic component.
\end{enumerate}
If $g = 3$ and some $k_j$ is odd, then $\Om\cM_g(\kap)$ is connected. If $g = 3$ and all $k_j$ are even, then $\Om\cM_g(\kap)$ has exactly two components: one odd component, and one hyperelliptic component which is also an even component. If $g = 2$, then $\Om\cM_g(\kap)$ is connected.
\end{thm}

\paragraph{\bf The $\GL^+(2,\R)$-action.} Let $\GL^+(2,\R)$ be the group of automorphisms of the vector space $\R^2$ with positive determinant. The standard $\R$-linear action of $\GL^+(2,\R)$ on $\C$ induces an action on $\Om\cM_g$, and this action preserves each stratum. The action of the subgroup $\SL(2,\R)$ also preserves the area-$1$ locus. For $s,t \in \R$, define
\be
g_t = \mat{e^t}{0}{0}{e^{-t}}, \quad u_s = \mat{1}{s}{0}{1}, \quad v_s = \mat{1}{0}{s}{1} .
\ee
The action of the diagonal subgroup $g_t$ is the {\em geodesic flow}, and the action of the unipotent subgroup $u_s$ is the {\em horocycle flow}. The subgroups $u_s$ and $v_s$ generate $\SL(2,\R)$.

In \cite{EMM:closures}, rigidity theorems are proven for $\GL^+(2,\R)$-orbit closures in strata. In particular, orbit closures are properly immersed suborbifolds locally defined by homogeneous $\R$-linear equations in period coordinates. Building off of this work, explicit full measure subsets of dense $\GL^+(2,\R)$-orbits in connected components of strata are given in \cite{Wri:field}. Let $\ol{\Q} \subset \C$ be the algebraic closure of $\Q$.

\begin{thm} \label{thm:Wdense} (\cite{Wri:field}, Corollary 1.3)
If the period coordinates of $(X,\om) \in \Om\cM_g(\kap)$ are linearly independent over $\ol{\Q} \cap \R$, then the $\GL^+(2,\R)$-orbit of $(X,\om)$ is dense in its connected component in $\Om\cM_g(\kap)$.
\end{thm}

See \cite{Wri:field} for a more general result. We only stated the special case that we will use. \\

\paragraph{\bf The absolute period foliation.} Given $(X,\om) \in \Om\cM_g(\kap)$, the projection from relative to absolute cohomology
\be
p : H^1(X,Z(\om);\C) \ra H^1(X;\C)
\ee
defines a holomorphic submersion on a neighborhood of $(X,\om)$ in $\Om\cM_g(\kap)$. The fibers have complex dimension $n - 1$ and form the leaves of a holomorphic foliation $\cA(\kap)$ called the {\em absolute period foliation} of $\Om\cM_g(\kap)$. Holomorphic $1$-forms on the same leaf of $\cA(\kap)$ have the same area. The action of $\GL^+(2,\R)$ sends leaves of $\cA(\kap)$ to leaves of $\cA(\kap)$.

Consider the finite cover
\be
\wt{\Om}\cM_g(\kap) \ra \Om\cM_g(\kap)
\ee
obtained by labelling the zeros $Z_1,\dots,Z_n$ of holomorphic $1$-forms in $\Om\cM_g(\kap)$. The foliation $\cA(\kap)$ lifts to a foliation of $\wt{\Om}\cM_g(\kap)$, which we call the {\em absolute period foliation} of $\wt{\Om}\cM_g(\kap)$. Consider the quotient vector spaces
\be
\C^n_0 = \C^n / \C(1,\dots,1), \quad \R^n_0 = \R^n / \R(1,\dots,1) ,
\ee
where $(1,\dots,1)$ is the constant vector. The labelling of zeros provides a canonical identification of the kernel of $p$ with $\C^n_0$. Note that for $v \in \C_0^n$, the difference between two components $v_i - v_j$ is well-defined. The corresponding element of the kernel of $p$ evaluates to $v_i - v_j$ on any path from $Z_j$ to $Z_i$. Associated to a real vector subspace $V \subset \C^n_0$ is then a foliation $\cA_V(\kap)$ of $\wt{\Om}\cM_g(\kap)$ whose leaves are contained in leaves of the absolute period foliation of $\wt{\Om}\cM_g(\kap)$. In the case where $v \in \R^n_0$ and $V = \C v$, the action of $\GL^+(2,\R)$ sends leaves of $\cA_V(\kap)$ to leaves of $\cA_V(\kap)$.

Let $L$ be the leaf of the absolute period foliation of $\wt{\Om}\cM_g(\kap)$ through $(X_0,\om_0)$. Given $q \in X$ and paths $\gam_j$ from $q$ to $Z_j$, the {\em relative period map}
\be
\rho(X,\om) = \left(\int_{\gam_1} \om, \dots, \int_{\gam_n} \om\right) \in \C^n_0
\ee
provides local coordinates on a neighborhood of $(X_0,\om_0)$ in $L$. The map $\rho$ does not depend on the choice of $q$, but different choices of paths may translate the components of $\rho$ by absolute periods, which are constant on $L$. These local coordinates give $L$ a translation structure modelled on $\C^n_0$, and in particular a straight-line flow associated to each nonzero $v \in \C^n_0$. This straight-line flow gives rise to a partially defined flow on $\wt{\Om}\cM_g(\kap)$ called a {\em Rel flow}. When $v \in \R^n_0$, this partially defined flow is called a {\em real Rel flow}. Orbits of Rel flows are not always well-defined for all time, since distinct zeros may collide in finite time. Rel flows will be denoted by $\Rel_{t v} (X,\om)$ with $t \in \R$ when well-defined. The map
\be
\wt{\Om}\cM_g(\kap) \times \R \ra \wt{\Om}\cM_g(\kap), \quad ((X,\om),t) \mapsto {\rm \Rel}_{tv} (X,\om)
\ee
is well-defined and continuous on an open subset of $\wt{\Om}\cM_g(\kap) \times \R$. We denote by
\be
I(v,\om) \subset \R
\ee
the maximal domain of definition of the map $t \mapsto \Rel_{tv}(X,\om)$. For real Rel flows, the only obstructions to having $I(v,\om) = \R$ come from horizontal saddle connections with distinct endpoints.

\begin{thm} \label{thm:domain} (\cite{BSW:horocycle}, Corollary 6.2) Fix a nonzero $v \in \R^n_0$. We have $t_0 \in I(v,\om)$ if and only if for all $1 \leq i \neq j \leq n$, $(X,\om)$ does not have a saddle connection $\gam$ from $Z_i$ to $Z_j$ such that $\int_\gam \om = t t_0 (v_i - v_j)$ for some $t \in [0,1]$.
\end{thm}

See also \cite{MW:cohomology} and \cite{McM:navigating}. In particular, if $(X,\om)$ does not have a horizontal saddle connection with distinct endpoints, then $\Rel_{\R v}(X,\om)$ is well-defined.

The horocycle flow commutes with real Rel flows, and the geodesic flow normalizes real Rel flows. For $v \in \R^n_0$ nonzero, $(X,\om) \in \wt{\Om}\cM_g(\kap)$, and $s,t \in \R$, we have
\begin{align*}
u_s {\rm \Rel}_v (X,\om) &= {\rm \Rel}_v u_s(X,\om) \\
g_t {\rm \Rel}_v (X,\om) &= {\rm \Rel}_{e^t v} g_t(X,\om)
\end{align*}
where in each case the left-hand side is well-defined if and only if the right-hand side is. \\

\paragraph{\bf The slit construction.} We describe a local surgery for changing the zero orders of a holomorphic $1$-form, or a collection of holomorphic $1$-forms on disjoint surfaces, without changing any absolute periods. See Sections 8-9 in \cite{EMZ:principal}, Section 5 in \cite{CDF:transfer}, and Section 3 in \cite{McM:navigating} for related discussions.

Let $(X_j,\om_j)$, $j = 1,\dots,n$, be a collection of holomorphic $1$-forms. Choose a nonzero $z \in \C$, choose points $p_j \in X_j$, and choose $1 \leq r_j \leq k_j + 1$, where $k_j \geq 0$ is the order of $\om_j$ at $p_j$. Suppose there are oriented geodesic segments $s_{j,1},\dots,s_{j,r_j}$ in $(X_j,\om_j)$ with $\int_{s_{j,k}} \om_j = z$, starting at $p_j$ and disjoint from $Z(\om_j)$ except possibly at $p_j$, and such that the counterclockwise angle around $p_j$ from $s_{j,1}$ to $s_{j,k}$ is $2\pi(k-1)$. We view the segments
\be
s_{1,1},s_{1,2},\dots,s_{1,r_1},s_{2,1},s_{2,2},\dots,s_{2,r_2},\dots,s_{n,1},s_{n,2},\dots,s_{n,r_n}
\ee
as being cyclically ordered. For notational simplicity, we rename the segments
\be
s_1,\dots,s_N,
\ee
respectively. We refer to the following surgery as {\em applying the slit construction} to the segments $s_1,\dots,s_N$. For $j = 1,\dots,N$, slit $s_j$ to obtain a pair of oriented segments $s_j^+,s_j^-$, corresponding to the left and right sides of $s_j$, respectively. Glue $s_j^+$ to $s_{j+1}^-$ isometrically and respecting the orientations, where indices are taken modulo $N$. The result is a connected topological surface, and the complex structure and the given holomorphic $1$-forms extend over the slits to give a holomorphic $1$-form $(X,\om)$.

The starting points of the segments yield a zero of $\om$ of order $-1 + \sum_{j=1}^n (k_j - r_j + 2)$, and the ending points of the segments yield a zero of order $-1 + \sum_{j=1}^n r_j$. In the special case where $n = 1$ and $1 < r_1 < k_1 + 1$, the zero $p_1$ of order $k_1$ is split into two zeros of orders $k_1 - r_1 + 1$ and $r_1 - 1$. In the special case where $k_j = 0$ for all $j$, the starting points of the segments yield a zero of $\om$ of order $n-1$, and the ending points of the segments yield another zero of $\om$ of order $n-1$.


\section{Dense real Rel flow orbits} \label{sec:densecriterion}

In this section, we give a criterion for the existence of dense real Rel flow orbits in $\SL(2,\R)$-orbit closures in strata of holomorphic $1$-forms with labelled zeros. We rely on an equidistribution result in \cite{For:pushforward} for geodesic pushforwards of horocycle arcs, which builds off of the equidistribution results in \cite{EM:stationary} and \cite{EMM:closures} for the action of the upper-triangular subgroup
\be
P = \left\{u_s g_t \; : \; s,t \in \R \right\} \subset \SL(2,\R) .
\ee
Throughout this section, $g \geq 2$ is fixed and $\kap = (k_1,\dots,k_n)$ is a fixed partition of $2g-2$.

\begin{thm} (Corollary 1.3, \cite{For:pushforward}) \label{thm:Fequi}
Let $\nu$ be the length measure on a horocycle arc $u_{[0,T]}(X,\om)$, normalized to have total mass $1$, and let $\mu$ be the $\SL(2,\R)$-invariant affine probability measure supported on the $\SL(2,\R)$-orbit closure of $(X,\om)$. There exists a subset $Z \subset \R_{>0}$ of zero upper density such that
\be
\lim_{\substack{t \ra +\infty \\ t \notin Z}} (g_t)_\ast \nu = \mu
\ee
in the weak-* topology.
\end{thm}

This result is useful in the study of real Rel flows, since real Rel flows commute with the horocycle flow and both flows are normalized by the geodesic flow. Our density criterion for real Rel flow orbits is the following.

\begin{thm} \label{thm:densecriterion}
Suppose $n > 1$, and fix a nonzero $v \in \R^n_0$. Fix $(X,\om) \in \wt{\Om}\cM_g(\kap)$, and let $\cM \subset \wt{\Om}\cM_g(\kap)$ be the $\SL(2,\R)$-orbit closure of $(X,\om)$. If $\Rel_{\R v} (X,\om)$ is contained in $\cM$ and the closure of $\Rel_{\R v} (X,\om)$ contains the horocycle through $(X,\om)$, then there exists $(Y,\eta) \in \cM$ such that $\Rel_{\R v}(Y,\eta)$ is dense in $\cM$.
\end{thm}

\begin{proof}
Fix $T > 0$, and let $\nu$ be the length measure on the horocycle arc $u_{[0,T]}(X,\om)$, normalized to have total mass $1$. Let $\mu$ be the $\SL(2,\R)$-invariant affine probability measure supported on $\cM$. By Theorem \ref{thm:Fequi}, there is a sequence $t_n \ra +\infty$ such that $(g_{t_n})_\ast \nu \ra \mu$ in the weak-* topology. The relation
\be
g_t u_s = u_{se^{2t}} g_t
\ee
implies that the support of $(g_{t_n})_\ast \nu$ is given by
\be
\supp((g_{t_n})_\ast \nu) = g_{t_n}u_{[0,T]}(X,\om) = u_{[0,Te^{2t_n}]}g_{t_n}(X,\om) .
\ee

By Theorem 2.1 in \cite{EMM:closures}, the $P$-orbit of $(X,\om)$ is dense in $\cM$. For $(Y,\eta) \in \cM$, there are sequences $s_m,t_m \in \R$ such that $u_{s_m}g_{t_m}(X,\om) \ra (Y,\eta)$. Since the map $((X^\pr,\om^\pr),t) \mapsto \Rel_{t v}(X^\pr,\om^\pr)$ is well-defined and continuous on an open subset of $\wt{\Om}\cM_g(\kap) \times \R$, for $t \in I(v,\eta)$ we have that
\be
u_{s_m}g_{t_m}{\rm Rel}_{te^{-t_m}v}(X,\om) = {\rm Rel}_{tv}u_{s_m}g_{t_m}(X,\om) \ra {\rm Rel}_{tv}(Y,\eta) .
\ee
Then since $\Rel_{\R v}(X,\om) \subset \cM$, and since $\cM$ is $\SL(2,\R)$-invariant and closed, we have $\Rel_{t v}(Y,\eta) \in \cM$.

For each nonempty open subset $\cU \subset \cM$, define
\be
\cS(v,\cU) = \left\{(Y,\eta) \in \cM \; : \; {\rm Rel}_{t v}(Y,\eta) \in \cU \text{ for some } t \in I(v,\eta)\right\} .
\ee
For each $t \in \R$, the map $(Y,\eta) \mapsto \Rel_{t v} (Y,\eta)$ is well-defined and continuous on an open subset of $\cM$, so $\cS(v,\cU)$ is open. Since the support of $\mu$ is $\cM$, we have $\mu(\cU) > 0$, and since $(g_{t_n})_\ast \nu \ra \mu$, for sufficiently large $n$ the support of $(g_{t_n})_\ast\nu$ intersects $\cU$, which means there exists $s_n \in [0,Te^{2t_n}]$ such that $u_{s_n} g_{t_n} (X,\om) \in \cU$. Clearly $\cU \subset \cS(v,\cU)$, and so
\be
u_{s_n} g_{t_n} (X,\om) \in \cS(v,\cU) .
\ee
By hypothesis, the closure of $\Rel_{\R v}(X,\om)$ contains the horocycle through $(X,\om)$. For $s \in \R$, there is a sequence $t_m^\pr \in \R$ such that ${\rm \Rel}_{t_m^\pr v}(X,\om) \ra u_s(X,\om)$. We then have
\be
u_{s_n}g_{t_n}{\rm \Rel}_{t_m^\pr v}(X,\om) \ra u_{s_n}g_{t_n}u_s(X,\om)
\ee
or equivalently,
\be
{\rm \Rel}_{t_m^\pr e^{t_n} v}u_{s_n}g_{t_n}(X,\om) \ra u_{se^{2t_n}}u_{s_n}g_{t_n}(X,\om) .
\ee
Since $s$ is arbitrary, the closure of $\Rel_{\R v}u_{s_n}g_{t_n}(X,\om)$ contains the horocycle through $u_{s_n}g_{t_n}(X,\om)$, and in particular contains the support of $(g_{t_n})_\ast \nu$. Since $u_{s_n}g_{t_n}(X,\om) \in \cS(v,\cU)$, by definition
\be
{\rm \Rel}_{\R v}u_{s_n}g_{t_n}(X,\om) \subset \cS(v,\cU)
\ee
so the closure of $\Rel_{\R v}u_{s_n}g_{t_n}(X,\om)$ is contained in the closure of $\cS(v,\cU)$. Then there is $n_0$ such that $\bigcup_{n \geq n_0} \supp((g_{t_n})_\ast \nu)$ is contained in the closure of $\cS(v,\cU)$. Since $(g_{t_n})_\ast \nu \ra \mu$, the union $\bigcup_{n \geq n_0} \supp((g_{t_n})_\ast\nu)$ is dense in $\cM$, thus $\cS(v,\cU)$ is dense in $\cM$.

Now choose a countable basis $\{\cU_m\}_{m=1}^\infty$ for the topology on $\cM$. For each $m$, we have that $\cS(v,\cU_m)$ is open and dense in $\cM$. It follows from Theorem \ref{thm:domain} that for any $(a,b) \subset \R$, the set of $(Y,\eta) \in \cM$ such that $(a,b) \subset I(v,\eta)$ is a dense open subset of $\cM$. Then the set of
\be
(Y,\eta) \in \bigcap_{m = 1}^\infty \cS(v,\cU_m)
\ee
such that $I(v,\eta) = \R$ is a countable intersection of dense open subsets of $\cM$, so by the Baire category theorem it is dense and in particular nonempty. For any such $(Y,\eta) \in \cM$, we have that $\Rel_{\R v}(Y,\eta)$ intersects $\cU_m$ for all $m$, and is therefore dense in $\cM$.
\end{proof}

In Section \ref{sec:construction}, we will construct explicit holomorphic $1$-forms satisfying the hypotheses of Theorem \ref{thm:densecriterion}. However, we emphasize that the dense real Rel flow orbits provided by Theorem \ref{thm:densecriterion} are not explicit. For instance, a holomorphic $1$-form $(X,\om)$ satisfying the hypotheses of Theorem \ref{thm:densecriterion} may have a periodic horizontal foliation, but in that case, the closure of $\Rel_{\R v}(X,\om)$ does not contain the $\SL(2,\R)$-orbit of $(X,\om)$.


\section{Twist parameters for the horocycle flow and real Rel flows} \label{sec:twist}

In this section, we study the horocycle flow and real Rel flows in the special case of holomorphic $1$-forms with a periodic horizontal foliation. See Section 6 in \cite{HW:Rel} for a more detailed discussion in the case of holomorphic $1$-forms with exactly $2$ zeros. Using the density criterion for real Rel flow orbits in Section \ref{sec:densecriterion}, and the density criterion for $\GL^+(2,\R)$-orbits in Section \ref{sec:background}, we reduce Theorem \ref{thm:denserealRel} to a problem of constructing certain holomorphic $1$-forms with a periodic horizontal foliation. Throughout this section, $g \geq 2$ is fixed and $\kap = (k_1,\dots,k_n)$ is a fixed partition of $2g-2$ with $n > 1$. \\

\paragraph{\bf Horizontal twist maps.} Suppose $(X,\om) \in \wt{\Om}\cM_g(\kap)$ has a periodic horizontal foliation, and let $C_1,\dots,C_m$ be the horizontal cylinders on $(X,\om)$. Let $h_j$ and $w_j$ be the height and circumference of $C_j$, respectively. Choose a saddle connection $\gam_j \subset C_j \cup Z(\om)$ crossing $C_j$ from bottom to top, and let $t_j \in \R / w_j \Z$ be the twist parameter of $C_j$ with respect to $\gam_j$. Let $\al_j \subset C_j$ be a closed geodesic with $\int_{\al_j} \om \in \R_{>0}$. Cutting $C_j$ along $\al_j$, twisting to the left by $a_j \in \R$, and regluing, changes the twist parameter of $C_j$ from $t_j$ to $t_j + a_j$. The lengths of the horizontal saddle connections and the heights and circumferences of the horizontal cylinders on $(X,\om)$ are unchanged in the process. In this way, we obtain a continuous map $\prod_{j=1}^m \R / w_j \Z \ra \wt{\Om}\cM_g(\kap)$. For convenience, we rescale the inputs to obtain a {\em horizontal twist map}
\be
\rho_\om : \R^m / \Z^m \ra \wt{\Om}\cM_g(\kap)
\ee
such that
\be
\rho_\om\left(\frac{t_1}{w_1},\dots,\frac{t_m}{w_m}\right) = (X,\om) .
\ee

\paragraph{\bf Horocycle and real Rel flows.} The horocycle flow preserves the lengths of horizontal saddle connections and the heights and circumferences of horizontal cylinders. The horocycle through $(X,\om)$ is obtained by twisting each horizontal cylinder at a rate proportional to its height. For $s \in \R$, we have
\be
u_s(X,\om) = \rho_\om\left(\frac{t_1 + sh_1}{w_1},\dots,\frac{t_m + sh_m}{w_m}\right) .
\ee
We will only need to know that $u_s(X,\om)$ is in the image of $\rho_\om$.

Next, fix a nonzero $v \in \R^n_0$. The real Rel flow associated to $v$ also preserves the heights and circumferences of horizontal cylinders. However, if $v_i - v_j \neq 0$, then the length of a horizontal saddle connection from $Z_i$ to $Z_j$ is not preserved. Suppose additionally that $v_i - v_j = 0$ whenever $Z_i$ and $Z_j$ are the endpoints of a horizontal saddle connection on $(X,\om)$. Then by Theorem \ref{thm:domain}, $\Rel_{\R v}(X,\om)$ is well-defined. Moreover, $\Rel_{\R v}(X,\om)$ is contained in the image of $\rho_\om$. Let $T_j$ and $B_j$ be the top and bottom boundaries of $C_j$, respectively. Choose functions
\be
t,b : \{1,\dots,m\} \ra \{1,\dots,n\}
\ee
such that $t(j)$ and $b(j)$ are the indices of a zero in $T_j$ and $B_j$, respectively. By our assumption, the differences $v_{t(j)} - v_{b(j)}$ do not depend on the choice of $t$ and $b$. The real Rel flow orbit through $(X,\om)$ is obtained by twisting $C_j$ at a rate proportional to the difference $v_{t(j)} - v_{b(j)}$. For $s \in \R$, we have
\be
{\rm \Rel}_{sv}(X,\om) = \rho_\om\left(\frac{t_1 + s(v_{t(1)} - v_{b(1)})}{w_1},\dots,\frac{t_m + s(v_{t(m)} - v_{b(m)})}{w_m}\right) .
\ee
Let $v_{\rm tw} = \left(\frac{v_{t(1)} - v_{b(1)}}{w_1},\dots,\frac{v_{t(m)} - v_{b(m)}}{w_m}\right)$, and consider the linear flow $\varphi_s(x) = x + sv_{\rm tw}$ on $\R^m/\Z^m$. Recall that the closure of the flow orbit through $x$ is the smallest subtorus of $\R^m/\Z^m$ containing $x$. In other words, the closure $\ol{\varphi_\R (x)}$ is given by $x + T_0$, where $T_0$ is the projection of the subspace of vectors $(s_1,\dots,s_m) \in \R^m$ whose components satisfy all homogeneous $\Q$-linear relations satisfied by $\frac{v_{t(1)} - v_{b(1)}}{w_1},\dots,\frac{v_{t(m)} - v_{b(m)}}{w_m}$. By applying $\rho_\om$ to a flow orbit, we obtain the following lemma.

\begin{lem} \label{lem:realReltwist}
Let $V \subset \R^m$ be the subspace of vectors $(s_1,\dots,s_m)$ whose components satisfy all homogeneous $\Q$-linear relations satisfied by $\frac{v_{t(1)} - v_{b(1)}}{w_1},\dots,\frac{v_{t(m)} - v_{b(m)}}{w_m}$. Let $T_0 \subset \R^m / \Z^m$ be the torus given by the projection of $V$, and let $T = \left(\frac{t_1}{w_1},\dots,\frac{t_m}{w_m}\right) + T_0$. Then the closure of $\Rel_{\R v}(X,\om)$ in $\wt{\Om}\cM_g(\kap)$ is $\rho_\om(T)$.
\end{lem}

\begin{cor} \label{cor:realRelhoro}
If $\frac{v_{t(1)} - v_{b(1)}}{w_1},\dots,\frac{v_{t(m)} - v_{b(m)}}{w_m}$ are linearly independent over $\Q$, then the closure of $\Rel_{\R v}(X,\om)$ in $\wt{\Om}\cM_g(\kap)$ contains the horocycle through $(X,\om)$.
\end{cor}

In light of Theorem \ref{thm:densecriterion} and Corollary \ref{cor:realRelhoro}, we are interested in finding conditions under which the $\GL^+(2,\R)$-orbit of a holomorphic $1$-form in $\wt{\Om}\cM_g(\kap)$ with a periodic horizontal foliation is dense in its connected component in $\wt{\Om}\cM_g(\kap)$. Note that for the hypothesis of Corollary \ref{cor:realRelhoro} to hold, it must in particular be the case that every horizontal cylinder is twisted along ${\rm Rel}_{\R v}(X,\om)$. \\

\paragraph{\bf Separatrix diagrams.} Associated to any holomorphic $1$-form $(X,\om)$ with a periodic horizontal foliation is a directed graph $\Gam = \Gam(X,\om)$ with a vertex for each zero of $\om$ and a directed edge for each horizontal saddle connection oriented from left to right. The orientation on $X$ gives $\Gam$ the structure of a ribbon graph, and the boundary components of the associated surface with boundary come in pairs, with one pair for each horizontal cylinder. The ribbon graph $\Gam$, together with the pairing of boundary components, is called a {\em separatrix diagram}. See Section 4 of \cite{KZ:components} for more discussion, where this concept was introduced.

As before, let $C_1,\dots,C_m$ be the horizontal cylinders on $(X,\om)$, let $\gam_j \subset C_j \cup Z(\om)$ be a saddle connection crossing $C_j$ from bottom to top, and let $T_j$ and $B_j$ be the top and bottom boundaries of $C_j$, respectively. Each saddle connection $\gam \in \Gam$ has a length $\ell_\gam \in \R_{>0}$. For each horizontal cylinder $C_j$, these lengths satisfy a homogeneous integral linear equation
\begin{equation} \label{eq:topbot}
\sum_{\gam \subset T_j} \ell_\gam = \sum_{\gam \subset B_j} \ell_\gam .
\end{equation}
The holomorphic $1$-form $(X,\om)$ is determined up to isomorphism by its separatrix diagram, the lengths of its horizontal saddle connections, and the heights and twist parameters of its horizontal cylinders. The height and twist parameter of $C_j$ are determined by $\int_{\gam_j} \om$, which lies in the upper half-plane $\H$. The positive solutions in $\R^\Gam$ to the equations in (\ref{eq:topbot}) form a simplicial cone $C(\Gam)$ of dimension $|\Gam| - m + 1$, where $|\Gam|$ is the number of horizontal saddle connections. In this way, we obtain a continuous map
\be
\sig_\Gam : \H^m \times C(\Gam) \ra \wt{\Om}\cM_g(\kap)
\ee
such that any holomorphic $1$-form in the image of $\sig_\Gam$ has a periodic horizontal foliation and the associated separatrix diagram is isomorphic to $\Gam$.

\begin{lem} \label{lem:QRindep}
There is a countable collection of $\R$-linear subspaces $V_j \subset \H^m \times C(\Gam)$ of positive codimension such that for $w \notin \bigcup_{j=1}^\infty V_j$, the $\GL^+(2,\R)$-orbit of $\sig_\Gam(w)$ is dense in its connected component in $\wt{\Om}\cM_g(\kap)$.
\end{lem}

\begin{proof}
The set of $w \in \H^m \times C(\Gam)$ such that the period coordinates of $\sig_\Gam(w)$ are linearly dependent over $\ol{\Q} \cap \R$ is a countable union of $\R$-linear subspaces of positive codimension. The lemma follows from Theorem \ref{thm:Wdense}.
\end{proof}

In particular, we can slightly perturb the lengths of the horizontal saddle connections in $(X,\om)$ and the twist parameters of the horizontal cylinders in $(X,\om)$ to obtain a holomorphic $1$-form with a dense $\GL^+(2,\R)$-orbit. Next, consider the length function
\be
\ell_j : C(\Gam) \ra \R_{>0}
\ee
recording the circumference of the horizontal cylinder $C_j$. The cone $C(\Gam)$ is an open subset of an $\R$-linear subspace of $\R^{\Gam}$, and $\ell_j$ is the restriction of a linear functional on this subspace. We can choose a basis for $H_1(X,Z(\om);\Z)$ of the form $\gam_1,\dots,\gam_m,\gam_{m+1},\dots,\gam_{2g+n-1}$ with $\gam_j \in \Gam$ for $j = m+1,\dots,2g+n-1$. With respect to this basis, the period coordinates of holomorphic $1$-forms in the image of $\sig_\Gam$ range over an open subset of $\H^m \times \R^{2g+n-m-1}$. It follows that two length functions $\ell_i$ and $\ell_j$ are collinear over $\R$ if and only if $C_i$ and $C_j$ are homologous.

\begin{lem} \label{lem:recip}
Suppose that $v \in \R^n_0$ is nonzero and that $(X,\om) \in \wt{\Om}\cM_g(\kap)$ has a periodic horizontal foliation with horizontal cylinders $C_1,\dots,C_m$ and separatrix diagram $\Gam$ satisfying the following properties.
\begin{enumerate}
    \item If $Z_i$ and $Z_j$ are the endpoints of a horizontal saddle connection, then $v_i - v_j = 0$.
    \item If $Z_i$ and $Z_j$ are zeros in the top and bottom boundaries, respectively, of the same horizontal cylinder, then $v_i - v_j \neq 0$.
    \item If $i \neq j$, then $C_i$ and $C_j$ are not homologous.
\end{enumerate}
Then there is a countable collection of zero sets of polynomials $U_j \subset \H^m \times C(\Gam)$ of positive codimension such that for $w \notin \bigcup_{j=1}^\infty U_j$, the closure of $\Rel_{\R v}(\sig_\Gam(w))$ contains the horocycle through $\sig_\Gam(w)$.
\end{lem}

\begin{proof}
The functions $1/\ell_j : \H^m \times C(\Gam) \ra \R$, $j = 1,\dots,m$, are reciprocals of linear functionals restricted to a nonempty open subset. Since $C_i$ and $C_j$ are not homologous for $i \neq j$, no two of these functions are collinear over $\R$. Since $v_{t(j)} - v_{b(j)} \neq 0$ for $j = 1,\dots,m$, the same holds for the functions $(v_{t(j)} - v_{b(j)})/\ell_j$. Then by Lemma 4.9 in \cite{Wri:cylinder}, the functions $(v_{t(j)} - v_{b(j)})/\ell_j$ are linearly independent over $\R$, and in particular over $\Q$. Then each choice of $q_1,\dots,q_m \in \Q$ not all zero determines a relation
\be
\sum_{j=1}^m q_j \frac{v_{t(j)} - v_{b(j)}}{\ell_j(w)} = 0
\ee
which holds only on a subspace of $\H^m \times C(\Gam)$ of positive codimension. After clearing denominators, the lemma follows from Corollary \ref{cor:realRelhoro}.
\end{proof}

\paragraph{\bf Density criterion.} We conclude this section with a criterion for the existence of a dense real Rel flow orbit in a connected component of $\wt{\Om}_1\cM_g(\kap)$, based on the existence of certain holomorphic $1$-forms with a periodic horizontal foliation.

\begin{thm} \label{thm:percriterion}
Fix a nonzero $v \in \R^n_0$, and let $\wt{\cC}_1$ be a connected component of $\wt{\Om}_1\cM_g(\kap)$. Suppose there exists $(X,\om) \in \wt{\cC}_1$ with a periodic horizontal foliation satisfying the following properties.
\begin{enumerate}
    \item If $Z_i$ and $Z_j$ are the endpoints of a horizontal saddle connection, then $v_i - v_j = 0$.
    \item If $Z_i$ and $Z_j$ are zeros in the top and bottom boundaries, respectively, of the same horizontal cylinder, then $v_i - v_j \neq 0$.
    \item Distinct horizontal cylinders are not homologous.
\end{enumerate}
Then there exists $(Y,\eta) \in \wt{\cC}_1$ such that $\Rel_{\R v} (Y,\eta)$ is dense in $\wt{\cC}_1$.
\end{thm}

\begin{proof}
Let $\Gam$ be the separatrix diagram associated to $(X,\om)$, and let $U_j, V_j \subset \H^m \times C(\Gam)$ be the subspaces as in Lemmas \ref{lem:QRindep} and \ref{lem:recip}. The complement of $\bigcup_{j=1}^\infty (U_j \cup V_j)$ in $\H^m \times C(\Gam)$ is dense in $\H^m \times C(\Gam)$ and invariant under scaling by $\R_{>0}$, so there is $(X^\pr,\om^\pr) \in \wt{\cC}_1$ in the image of this complement under $\sig_\Gam$. The closure of $\Rel_{\R v}(X^\pr,\om^\pr)$ in $\wt{\Om}\cM_g(\kap)$ contains the horocycle through $(X^\pr,\om^\pr)$, and the $\SL(2,\R)$-orbit of $(X^\pr,\om^\pr)$ is dense in $\wt{\cC}_1$. Thus, by Theorem \ref{thm:densecriterion} there exists $(Y,\eta) \in \wt{\cC}_1$ such that $\Rel_{\R v}(Y,\eta)$ is dense in $\wt{\cC}_1$.
\end{proof}


\section{Periodic horizontal foliations with non-homologous cylinders} \label{sec:construction}

In this section, we prove Theorem \ref{thm:denserealRel} and Theorem \ref{thm:denseV}. Throughout, we fix $g \geq 2$, a partition $\kap = (k_1,\dots,k_n)$ of $2g-2$ with $n > 1$, and a nonzero $v \in \R^n_0$. To prove Theorem \ref{thm:denserealRel}, by Theorem \ref{thm:percriterion} it is enough to establish the following.

\begin{thm} \label{thm:constr}
Let $\wt{\cC}$ be a connected component of $\wt{\Om}\cM_g(\kap)$. There exists $(X,\om) \in \wt{\cC}$ with a periodic horizontal foliation satisfying the following properties.
\begin{enumerate}
    \item If $Z_i$ and $Z_j$ are the endpoints of a horizontal saddle connection, then $v_i - v_j = 0$.
    \item If $Z_i$ and $Z_j$ are zeros in the top and bottom boundaries, respectively, of the same horizontal cylinder, then $v_i - v_j \neq 0$.
    \item Distinct horizontal cylinders are not homologous.
\end{enumerate}
\end{thm}

Before beginning the proofs of Theorems \ref{thm:constr} and \ref{thm:denseV}, we refer to Figure \ref{fig:slittori11} for an example in the simplest case where $\kap = (1,1)$, which illustrates the basic idea of our constructions. In the top image, the holomorphic $1$-form $(X,\om)$ is presented as a connected sum of tori $T_1,T_2$ with periodic horizontal and vertical foliations. In the bottom image, the horizontal cylinders $C_1,C_2,C_3$ and the vertical cylinders $D_1,D_2,D_3$ are labelled. If $w_j$ is the circumference of $C_j$ and $w_j^\pr$ is the circumference of $D_j$, then $w_3 = w_1 + w_2$ and $w_3^\pr = w_1^\pr + w_2^\pr$. Let $z_1 \in \C$ be the integral of $\om$ along one of the slits from $Z_2$ to $Z_1$. The horizontal foliation of $(X,\om)$ clearly satisfies the conditions of Theorem \ref{thm:constr}, thus Theorem \ref{thm:denserealRel} holds in the case where $\kap = (1,1)$. Regarding Theorem \ref{thm:denseV}, note that the vertical foliation of $(X,\om)$ satisfies analogous conditions. In other words, the horizontal foliation of the rotated holomorphic $1$-form $(X,-i\om)$ also satisfies the conditions of Theorem \ref{thm:constr}. We can choose $w_1,w_2,w_1^\pr,w_2^\pr,z_1$ such that the reciprocals $\frac{1}{w_1},\frac{1}{w_2},\frac{1}{w_3}$ are linearly independent over $\Q$, the reciprocals $\frac{1}{w_1^\pr},\frac{1}{w_2^\pr},\frac{1}{w_3^\pr}$ are linearly independent over $\Q$, and the period coordinates $w_1,w_2,iw_1^\pr,iw_2^\pr,z_1$ are linearly independent over $\ol{\Q} \cap \R$. Then by Corollary \ref{cor:realRelhoro}, the closure of ${\rm Rel}_{\R v}(X,\om)$ contains $u_\R (X,\om)$, and similarly the closure of ${\rm Rel}_{i\R v}(X,\om)$ contains $v_\R (X,\om)$. Then the closure of the leaf of $\cA(\kap)$ through $(X,\om)$ is $\SL(2,\R)$-invariant. By Theorem \ref{thm:Wdense}, the $\SL(2,\R)$-orbit of $(X,\om)$ is dense in $\wt{\Om}_1\cM_2(1,1)$. Thus, Theorem \ref{thm:denseV} holds in the case where $\kap = (1,1)$.

We will split the proof of Theorem \ref{thm:constr} into 5 cases. Each case will build off of the construction in the first case, which is based on repeated use of the slit construction from Section \ref{sec:background}. Our construction is not the simplest possible, but it does have some advantages. In particular, the holomorphic $1$-forms we construct will also have periodic vertical foliations, which we will analyze along the way, in order to prove Theorem \ref{thm:denseV} afterward.

\begin{figure}
    \centering
    \includegraphics[width=0.4\textwidth]{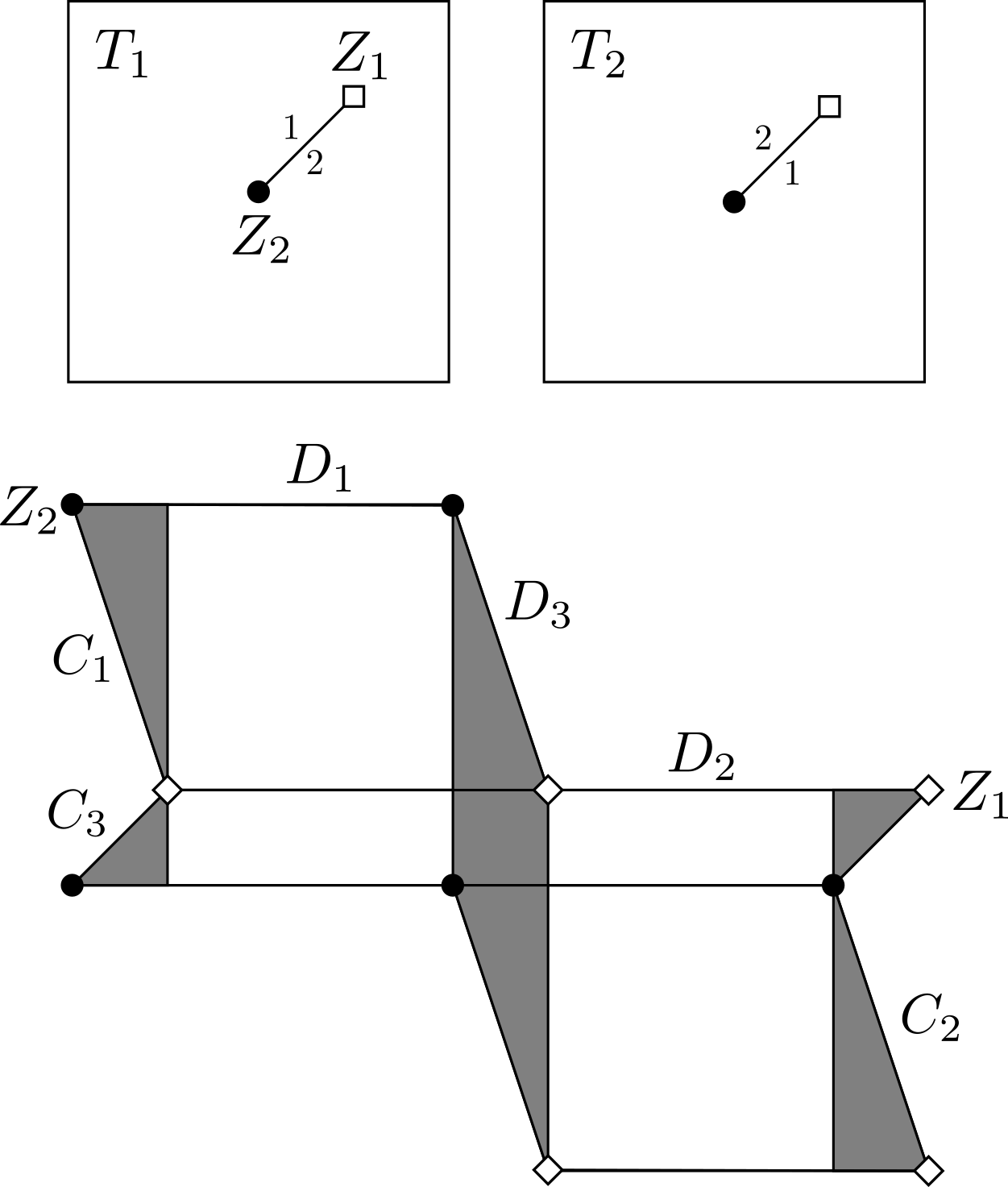}
    \caption{Two presentations of the same $(X,\om) \in \wt{\Om}\cM_2(1,1)$ satisfying the conditions of Theorem \ref{thm:constr}.}
    \label{fig:slittori11}
\end{figure}

\begin{proof} (of Theorem \ref{thm:constr})
The proof is given in $5$ cases. \\

\paragraph{\bf Case 1:} We make the following assumptions.
\begin{itemize}
    \item The components of $v$ are distinct.
    \item For all $j$, we have $k_j \leq g - 1$.
    \item If $k_j$ is even for all $j$, then the associated parity of spin structure is $g \hmod 2$.
    \item If $g \geq 3$, then the image of $\wt{\cC}$ in $\Om\cM_g(\kap)$ is a nonhyperelliptic component.
\end{itemize}
For $1 \leq j \leq g$, let
\be
T_j = (\C/\Lam_j,dz)
\ee
be a flat torus with periodic horizontal and vertical foliations, and let $\pi_j : \C \ra T_j$ be the associated projection. Let $\alpha_j \subset T_j$ be a closed geodesic with $\int_{\al_j} dz \in \R_{>0}$, and let $\beta_j \subset T_j$ be a closed geodesic with $\int_{\bet_j} dz \in i\R_{>0}$. Since $k_j \leq g - 1$ for $1 \leq j \leq n$, and $k_1 + \cdots + k_n = 2g-2$, there is a maximal $r$ such that $1 \leq r \leq n-1$ and
\be
k_1 + \cdots + k_r \leq g - 1 ,
\ee
and a minimal $r^\pr$ such that $2 \leq r^\pr \leq n$ and
\be
k_{r^\pr} + \cdots + k_n \leq g - 1 .
\ee
We have
\be
r^\pr = \begin{cases}
r + 1, \quad \text{ if } \; k_1 + \cdots + k_r = g-1 , \\
r + 2, \quad \text{ if } \; k_1 + \cdots + k_r < g-1 .
\end{cases}
\ee
Let
\be
A = \min_{1 \leq j \leq g} \int_{\al_j} dz > 0, \quad B = \min_{1 \leq j \leq g} -i \int_{\bet_j} dz > 0 . 
\ee
Choose $x_j,y_j \in \R$ for $1 \leq j \leq r$ and $r^\pr \leq j \leq n$, such that
\begin{align}
-\frac{A}{2} < x_n < \cdots < x_{r^\pr} < & \; 0 < x_r < \cdots < x_1 < \frac{A}{2} \label{eq:smallx} \\
-\frac{B}{2} < y_n < \cdots < y_{r^\pr} < & \; 0 < y_r < \cdots < y_1 < \frac{B}{2} \label{eq:smally} \\
-\frac{y_{r^\pr}}{x_{r^\pr}} < \cdots < -\frac{y_n}{x_n} < & \; 0 < \frac{y_1}{x_1} < \cdots < \frac{y_r}{x_r} \label{eq:slope}
\end{align}
and let $z_j = x_j + iy_j \in \C$. If $r^\pr = r + 2$, let $x_{r+1} = y_{r+1} = z_{r+1} = 0$. For $1 \leq j \leq r$ and $r^\pr \leq j \leq n$, let
\be
I_j = [0,z_j] = [0,1] \cdot z_j \subset \C
\ee
be the oriented line segment from $0$ to $z_j$. For $1 \leq j \leq r$ and $r^\pr \leq j \leq n$, and for $1 \leq k \leq g$, the projection $\pi_k(I_j) \subset T_k$ is an embedded oriented segment. Moreover, the segments
\be
\pi_k(I_1), \dots, \pi_k(I_r), \pi_k(I_n), \dots, \pi_k(I_{r^\pr})
\ee
are disjoint except at their common starting point $\pi_k(0)$, and are in counterclockwise order around $\pi_k(0)$. By replacing $\al_k$ and $\bet_k$ with parallel closed geodesics, we may assume that $\al_k$ and $\bet_k$ are disjoint from all of the segments $\pi_k(I_j)$.

Below, we will repeatedly apply the slit construction from Section \ref{sec:background} to the tori $T_1,\dots,T_g$. Each application of the slit construction will involve a subset of the segments $\pi_1(I_j),\dots,\pi_g(I_j)$ for some $j$. The segments involved will always be embedded, disjoint except at a common starting point, and will not contain any zeros except at their starting point. Our construction takes place in $r + (n - r^\pr + 1)$ steps. If $r^\pr = r + 1$, then there are $n$ steps, and if $r^\pr = r + 2$, then there are $n - 1$ steps.

In step $1$, we apply the slit construction to the segments
\be
\pi_1(I_1),\dots,\pi_{k_1+1}(I_1)
\ee
to obtain a holomorphic $1$-form $(Y_1,\eta_1)$ on a surface of genus $k_1 + 1$, with a zero $Z_1$ of order $k_1$ coming from the points $\pi_1(z_1),\dots,\pi_{k_1+1}(z_1)$, and a zero $Z_1^\pr$ of order $k_1$ coming from the points $\pi_1(0),\dots,\pi_{k_1+1}(0)$. In step $2$, we apply the slit construction to the segments
\be
\pi_{k_1+1}(I_2),\dots,\pi_{k_1+k_2+1}(I_2)
\ee
to obtain a holomorphic $1$-form $(Y_2,\eta_2)$ on a surface of genus $k_1 + k_2 + 1$. The order of the zero $Z_1$ from step $1$ is unchanged. Additionally, $\eta_2$ has a zero $Z_2$ of order $k_2$ coming from the points $\pi_{k_1+1}(z_2),\dots,\pi_{k_1+k_2+1}(z_2)$, and a zero $Z_2^\pr$ of order $k_1 + k_2$ coming from the points $Z_1^\pr,\pi_{k_1+2}(0),\dots,\pi_{k_1+k_2+1}(0)$. We continue in this way through step $r$. In step $s$, where $2 \leq s \leq r$, we apply the slit construction to the segments
\be
\pi_{k_1+\cdots+k_{s-1}+1}(I_s),\dots,\pi_{k_1+\cdots+k_s+1}(I_s)
\ee
to obtain a holomorphic $1$-form $(Y_s,\eta_s)$ on a surface of genus $k_1 + \cdots + k_s + 1$. The orders of the zeros $Z_1,\dots,Z_{s-1}$ from steps $1$ through $s-1$ are unchanged. Additionally, $\eta_s$ has a zero $Z_s$ of order $k_s$ coming from the points
\be
\pi_{k_1+\cdots+k_{s-1}+1}(z_s),\dots,\pi_{k_1+\cdots+k_s+1}(z_s) ,
\ee
and a zero $Z_s^\pr$ of order $k_1 + \cdots + k_s$ coming from the points
\be
Z_{s-1}^\pr,\pi_{k_1+\cdots+k_{s-1}+2}(0),\dots,\pi_{k_1+\cdots+k_s+1}(0) .
\ee

In step $r + 1$, we apply the slit construction to the segments
\be
\pi_g(I_n), \pi_{g-1}(I_n), \dots, \pi_{g-k_n}(I_n)
\ee
to obtain a holomorphic $1$-form $(Y_{r+1},\eta_{r+1})$. If
\be
k_1 + \cdots + k_r + 1 < g - k_n,
\ee
then $Y_{r+1}$ has genus $k_n + 1$. In this case, $\eta_{r+1}$ has a zero $Z_n$ of order $k_n$ coming from the points $\pi_g(z_n),\pi_{g-1}(z_n),\dots,\pi_{g-k_n}(z_n)$, and a zero $Z_n^\pr$ of order $k_n$ coming from the points $\pi_g(0)$, $\pi_{g-1}(0),\dots,\pi_{g-k_n}(0)$. Otherwise, $Y_{r+1}$ has genus $g$. In this case, for
\be
g - k_n + 1 \leq j \leq k_1 + \cdots + k_r + 1,
\ee
by (\ref{eq:slope}) the counterclockwise angle from $\pi_j(I_n)$ to $\pi_{j-1}(I_n)$ around their common starting point in $(Y_r,\eta_r)$ is $2\pi$. The orders of the zeros $Z_1,\dots,Z_r$ from steps $1$ through $r$ are unchanged. Additionally, $\eta_{r+1}$ has a zero $Z_n$ of order $k_n$ coming from the points $\pi_g(z_n),\pi_{g-1}(z_n),\dots,\pi_{g-k_n}(z_n)$, and a zero $Z_n^\pr$ of order
\be
(2g - 2) - (k_1 + \cdots + k_r + k_n)
\ee
coming from some of the starting points of the slits. We continue in this way through step $r + (n-r^\pr+1)$. In step $r+s$, where $2 \leq s \leq n-r^\pr+1$, we apply the slit construction to the segments
\be
\pi_{g-(k_{n-s+2}+\cdots+k_n)}(I_{n-s+1}),\pi_{g-(k_{n-s+2}+\cdots+k_n+1)}(I_{n-s+1}),\dots,\pi_{g-(k_{n-s+1}+\cdots+k_n)}(I_{n-s+1})
\ee
to obtain a holomorphic $1$-form $(Y_{r+s},\eta_{r+s})$. If
\be
k_1 + \cdots + k_r + 1 < g - (k_{n-s+1} + \cdots + k_n) ,
\ee
then $Y_{r+s}$ has genus $k_{n-s+1} + \cdots + k_n + 1$. In this case, the orders of the zeros $Z_n,Z_{n-1},\dots$, $Z_{n-s+2}$ from steps $r+1$ through $r+s-1$ are unchanged. Additionally, $\eta_{r+s}$ has a zero $Z_{n-s+1}$ of order $k_{n-s+1}$ coming from the points
\be
\pi_{g-(k_{n-s+2}+\cdots+k_n)}(z_{n-s+1}),\pi_{g-(k_{n-s+2}+\cdots+k_n+1)}(z_{n-s+1}),\dots,\pi_{g-(k_{n-s+1}+\cdots+k_n)}(z_{n-s+1}) ,
\ee
and a zero $Z_{n-s+1}^\pr$ of order $k_n + k_{n-1} + \cdots + k_{n-s+1}$ coming from the points
\be
Z_{n-s+2}^\pr,\pi_{g-(k_{n-s+2}+\cdots+k_n)}(0),\pi_{g-(k_{n-s+2}+\cdots+k_n+1)}(0),\dots,\pi_{g-(k_{n-s+1}+\cdots+k_n)}(0) .
\ee
Otherwise, $Y_{r+s}$ has genus $g$. In this case, for
\be
g - (k_{n-s+1} + \cdots + k_n) + 1 \leq j \leq \min\{g - (k_{n-s+2} + \cdots + k_n) + 1, k_1 + \cdots + k_r + 1\},
\ee
by (\ref{eq:slope}) the counterclockwise angle from $\pi_j(I_{n-s+1})$ to $\pi_{j-1}(I_{n-s+1})$ around their common starting point in $(Y_{r+s-1},\eta_{r+s-1})$ is $2\pi$. The orders of the zeros $Z_1,\dots,Z_r,Z_{n-s+2},\dots,Z_n$ from steps $1$ through $r+s-1$ are unchanged. Additionally, $\eta_{r+s}$ has a zero $Z_{n-s+1}$ of order $k_{n-s+1}$ coming from the points
\be
\pi_{g-(k_{n-s+2}+\cdots+k_n)}(z_{n-s+1}),\pi_{g-(k_{n-s+2}+\cdots+k_n+1)}(z_{n-s+1}),\dots,\pi_{g-(k_{n-s+1}+\cdots+k_n)}(z_{n-s+1}),
\ee
and a zero $Z_{n-s+1}^\pr$ of order
\be
2g - 2 - (k_1 + \cdots + k_r + k_{n-s+1} + \cdots + k_n)
\ee
coming from some of the starting points of the slits.

Now let
\begin{equation} \label{eq:X0w0}
(X_0,\omega_0) = (Y_{r+(n-r^\pr+1)},\eta_{r+(n-r^\pr+1)}) .
\end{equation}
Note that $X_0$ is connected and of genus $g$, since either $r^\pr = r + 1$, in which case $k_1 + \cdots + k_r + 1 = g$, or $r^\pr = r + 2$ and $k_{r+1} \leq g - 1$, in which case
\be
k_1 + \cdots + k_r + 1 \geq g - (k_{r^\pr} + \cdots + k_n) .
\ee
If $r^\pr = r+1$, then the points on $X_0$ arising from the starting points of the slits are not zeros of $\om_0$, and in particular $Z_{r^\pr}^\pr$ is not a zero of $\om_0$. If $r^\pr = r+2$, then letting $Z_{r+1} = Z_{r^\pr}^\pr$, we have that $Z_{r+1}$ is a zero of order $k_{r+1}$ coming from some of the starting points of the slits. Thus, the zeros $Z_1,\dots,Z_n$ of $\om_0$ have orders $k_1,\dots,k_n$, respectively, so
\be
(X_0,\om_0) \in \wt{\Om}\cM_g(\kap) .
\ee

Next, we describe the horizontal cylinders on $(X_0,\om_0)$ and we verify that $(X_0,\om_0)$ satisfies conditions (1), (2), and (3) in Theorem \ref{thm:constr}. For the purpose of proving Theorem \ref{thm:denseV}, we also describe the vertical cylinders on $(X_0,\om_0)$ and we verify that $(X_0,-i\om_0)$ satisfies conditions (1), (2), and (3) in Theorem \ref{thm:constr}. The horizontal and vertical foliations of $(X_0,\om_0)$ are clearly periodic. Abusing notation, for $1 \leq j \leq g$, denote by $T_j \subset X_0$ the open subset arising from the complement of the slits in $T_j$.

For $1 \leq j \leq g$, there is a unique horizontal cylinder $C_j$ contained in $T_j$, and there is a unique vertical cylinder $D_j$ contained in $T_j$. We have $\al_j \subset C_j$ and $\bet_j \subset D_j$. If $j \leq k_1 + \cdots + k_r + 1$, then let $i_j$ be minimal such that $1 \leq i_j \leq r$ and
\be
j \leq k_1 + \cdots + k_{i_j} + 1,
\ee
and otherwise, let $i_j = 0$. If $j \geq g - (k_{r^\pr} + \cdots + k_n)$, then let $i_j^\pr$ be maximal such that $r^\pr \leq i_j^\pr \leq n$ and
\be
j \geq g - (k_{i_j^\pr} + \cdots + k_n),
\ee
and otherwise, let $i_j^\pr = 0$. Letting $h_j$ be the height of $C_j$ and $h_j^\pr$ the height of $D_j$, we have
\be
h_j = -i\int_{\beta_j} \om_0 - (y_{i_j} - y_{i_j^\pr}), \quad h_j^\pr = \int_{\alpha_j} \om_0 - (x_{i_j} - x_{i_j^\pr}) .
\ee
Letting $w_j$ be the circumference of $C_j$ and $w_j^\pr$ the circumference of $D_j$, we have
\be
w_j = \int_{\alpha_j} \om_0, \quad w_j^\pr = -i \int_{\beta_j} \om_0 .
\ee

For $1 \leq j \leq r$, there is a unique horizontal cylinder $C_{g+j} \subset X_0$ such that $C_{g+j}$ intersects $T_k$ if and only if $1 \leq k \leq k_1 + \cdots + k_j + 1$, and there is a unique vertical cylinder $D_{g+j} \subset X_0$ such that $D_{g+j}$ intersects $T_k$ if and only if $1 \leq k \leq k_1 + \cdots + k_j + 1$. Letting $h_{g+j}$ be the height of $C_{g+j}$ and $h_{g+j}^\pr$ the height of $D_{g+j}$, we have
\be
h_{g+j} = y_j - y_{j+1}, \quad h_{g+j}^\pr = x_j - x_{j+1} .
\ee
Letting $w_{g+j}$ be the circumference of $C_{g+j}$ and $w_{g+j}^\pr$ the circumference of $D_{g+j}$, we have
\begin{align*}
w_{g+j} &= w_1 + \cdots + w_{k_1 + \cdots + k_j + 1}, \\
w_{g+j}^\pr &= w_1^\pr + \cdots + w_{k_1 + \cdots + k_j + 1}^\pr .
\end{align*}

For $r^\pr \leq j \leq n$, there is a unique horizontal cylinder $C_{g+j-1} \subset X_0$ such that $C_{g+j-1}$ intersects $T_k$ if and only if $g - (k_j+\cdots+k_n) \leq k \leq g$, and there is a unique vertical cylinder $D_{g+j-1} \subset X_0$ such that $D_{g+j-1}$ intersects $T_k$ if and only if $g - (k_j+\cdots+k_n) \leq k \leq g$. Letting $h_{g+j-1}$ be the height of $C_{g+j-1}$ and $h_{g+j-1}^\pr$ the height of $D_{g+j-1}$, we have
\be
h_{g+j-1} = y_{j-1} - y_j, \quad h_{g+j-1}^\pr = x_{j-1} - x_j .
\ee
Letting $w_{g+j-1}$ be the circumference of $C_{g+j-1}$ and $w_{g+j-1}^\pr$ the circumference of $D_{g+j-1}$, we have
\begin{align*}
w_{g+j-1} &= w_{g-(k_j+\cdots+k_n)} + \cdots + w_g, \\
w_{g+j-1}^\pr &= w_{g-(k_j+\cdots+k_n)}^\pr + \cdots + w_g^\pr .
\end{align*}
If $r^\pr = r+1$, the cylinders $C_{g+r}$ and $D_{g+r}$ have been mentioned twice above. The cylinders $C_1,\dots,C_{g+n-1}$ and $D_1,\dots,D_{g+n-1}$ account for all of the horizontal and vertical cylinders on $(X_0,\om_0)$.

For $1 \leq j \leq g$, the top boundary of $C_j$ consists of a single saddle connection. If $i_j^\pr \neq 0$, this saddle connection joins $Z_{i_j^\pr}$ to itself, and otherwise, this saddle connection joins $Z_{r+1}$ to itself. The bottom boundary of $C_j$ also consists of a single saddle connection. If $i_j \neq 0$, this saddle connection joins $Z_{i_j}$ to itself, and otherwise, this saddle connection joins $Z_{r+1}$ to itself. At least one of $i_j$ and $i_j^\pr$ is nonzero. If $i_j = 0$, then $r^\pr = r + 2$ and $i_j^\pr \geq r^\pr > r + 1$. If $i_j^\pr = 0$, then $i_j \leq r < r + 1$. If both $i_j$ and $i_j^\pr$ are nonzero, then $i_j \leq r < r^\pr \leq i_j^\pr$.  Thus, the zeros in the top and bottom boundaries of $C_j$ are distinct. The same holds for the left and right boundaries of $D_j$, respectively.

For $1 \leq j \leq r$, the top boundary of $C_{g+j}$ consists of $k_j + 1$ saddle connections joining $Z_j$ to itself. If $j = 1$, these saddle connections lie in the bottom boundaries of
\be
C_1,\dots,C_{k_1+1},
\ee
in cyclic order from left to right, and otherwise, these saddle connections lie in the bottom boundaries of
\be
C_{g+j-1},C_{k_1+\cdots+k_{j-1}+2},\dots,C_{k_1+\cdots+k_j+1},
\ee
in cyclic order from left to right. If $j < r$, the bottom boundary of $C_{g+j}$ consists of a single saddle connection joining $Z_{j+1}$ to itself and lying in the top boundary of $C_{g+j+1}$. If $r^\pr = r + 2$, the bottom boundary of $C_{g+r}$ consists of $g - (k_{r^\pr} + \cdots + k_n)$ saddle connections joining $Z_{r+1}$ to itself and lying in the top boundaries of
\be
C_1,\dots,C_{g - (k_{r^\pr} + \cdots + k_n)-1},C_{g+r+1},
\ee
in cyclic order from left to right. Analogous statements hold for the left and right boundaries of $D_{g+j}$.

For $r^\pr \leq j \leq n$, the bottom boundary of $C_{g+j-1}$ consists of $k_j+1$ saddle connections joining $Z_j$ to itself. If $j = m$, these saddle connections lie in the top boundaries of
\be
C_{g-k_n},\dots,C_g,
\ee
in cyclic order from left to right, and otherwise, these saddle connections lie in the top boundaries of
\be
C_{g-(k_j+\cdots+k_n)},\dots,C_{g-(k_{j+1}+\cdots+k_n)-1},C_{g+j},
\ee
in cyclic order from left to right. If $r^\pr < j$, the top boundary of $C_{g+j-1}$ consists of a single saddle connection joining $Z_{j-1}$ to itself and lying in the bottom boundary of $C_{g+j-2}$. If $r^\pr = r + 2$, the top boundary of $C_{g+r+1}$ consists of $g - (k_1 + \cdots + k_r)$ saddle connections joining $Z_{r+1}$ to itself and lying in the bottom boundaries of
\be
C_{g+r},C_{k_1+\cdots+k_r+2},\dots,C_g,
\ee
in cyclic order from left to right. If $r^\pr = r + 1$, the top boundary of $C_{g+r}$ was described in the previous paragraph, and the bottom boundary of $C_{g+r}$ was described in this paragraph. Again, analogous statements hold for the left and right boundaries of $D_{g+j-1}$. We have verified that $(X_0,\om_0)$ satisfies conditions (1) and (2), and we have verified that $(X_0,-i\om_0)$ satisfies conditions (1) and (2).

\begin{figure}
    \centering
    \includegraphics[width=0.8\textwidth]{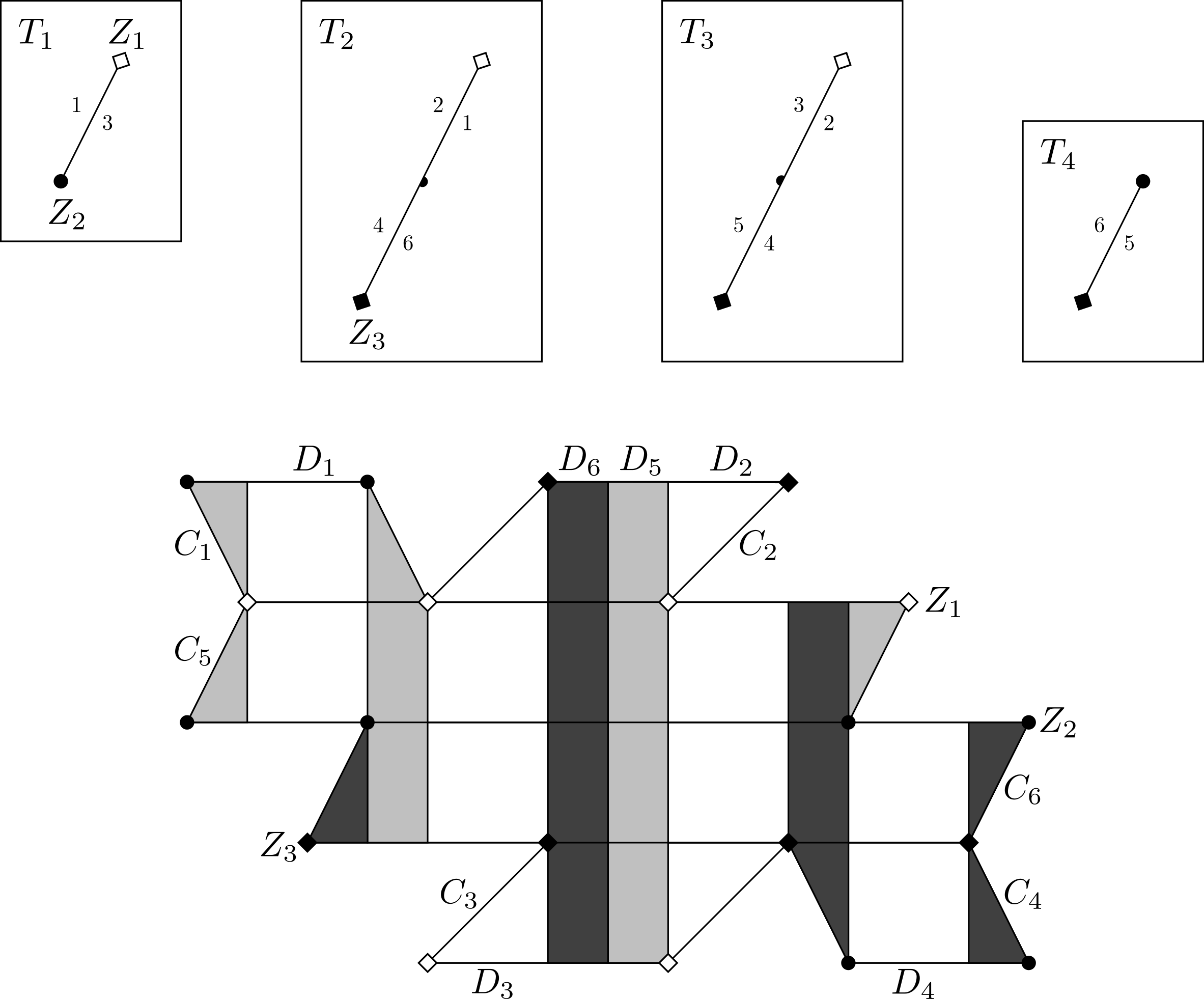}
    \caption{Two presentations of the same $(X,\om) \in \wt{\Om}\cM_4(2,2,2)$ with $\phi(\om) = 0 \hmod 2$ arising from the construction in Case 1.}
    \label{fig:slittori222}
\end{figure}

For $1 \leq j \leq g + n - 1$, let $\al_j \subset C_j$ be a closed geodesic with $\int_{\al_j} \om_0 \in \R_{>0}$, let $\bet_j \subset D_j$ be a closed geodesic with $\int_{\bet_j} \om_0 \in i\R_{>0}$, and let
\be
a_j = [\alpha_j] \in H_1(X_0;\Z), \quad b_j = [\beta_j] \in H_1(X_0;\Z) .
\ee
Clearly, $\{a_j,b_j\}_{j=1}^g$ is a symplectic basis for $H_1(X_0;\Z)$, and
\be
\ind(\al_j) = \ind(\bet_j) = 0 ,
\ee
so in the case where all $k_j$ are even, the parity of the associated spin structure is given by
\be
\phi(\om_0) = g \hmod 2 .
\ee
For $1 \leq j \leq r$, we have
\begin{align*}
a_{g+j} &= a_1 + \dots + a_{k_1 + \cdots + k_j + 1} \\
b_{g+j} &= b_1 + \dots + b_{k_1 + \cdots + k_j + 1}
\end{align*}
and for $r^\pr \leq j \leq m$, we have
\begin{align*}
a_{g+j-1} &= a_{g - (k_j + \cdots + k_n)} + \cdots + a_g \\
b_{g+j-1} &= b_{g - (k_j + \cdots + k_n)} + \cdots + b_g
\end{align*}
by considering algebraic intersection numbers with $a_k,b_k$, $1 \leq k \leq g$. In particular, the homology classes $a_1,\dots,a_{g+n-1}$ and $b_1,\dots,b_{g+n-1}$ are distinct. We have verified that $(X_0,\om_0)$ satisfies condition (3), and we have verified that $(X_0,-i\om_0)$ satisfies condition (3).

Lastly, suppose that $g \geq 3$ and $\kap = (g-1,g-1)$, and suppose that the image of $(X_0,\om_0)$ in $\Om\cM_g(g-1,g-1)$ lies in the hyperelliptic component. Then the hyperelliptic involution $\tau$ preserves each horizontal cylinder on $(X_0,\om_0)$, and so preserves $T_j$ for all $j$. Since $g \geq 3$, each $T_j$ shares a unique saddle connection with $T_{j-1}$ and a unique saddle connection with $T_{j+1}$, where indices are taken modulo $g$. Each of these saddle connections must be preserved by $\tau$. However, this is impossible since $\tau^\ast \om_0 = -\om_0$.

Letting $(X,\omega) = (X_0,\omega_0)$, we are done with Case 1. Figure \ref{fig:slittori11} shows an example in the simplest case $\kap = (1,1)$, and Figure \ref{fig:slittori222} shows an example in the case $\kap = (2,2,2)$, with the presentation as a connected sum of tori as well as the horizontal and vertical cylinder decompositions. Figure \ref{fig:slittori211} shows an additional example in the case $\kap = (2,1,1)$. \\

\begin{figure}
    \centering
    \includegraphics[width=0.6\textwidth]{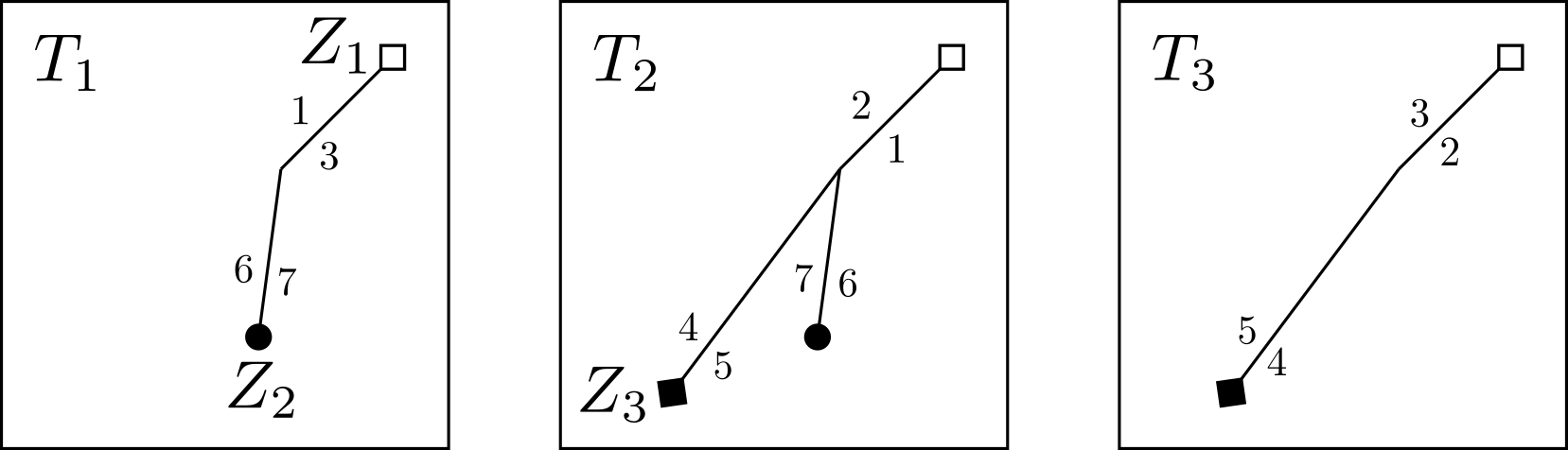}
    \caption{An example $(X,\om) \in \wt{\Om}\cM_3(2,1,1)$ arising from the construction in Case 1. The starting points of the slits (unmarked) are not zeros.}
    \label{fig:slittori211}
\end{figure}

\paragraph{\bf Case 2:} We make the following assumptions.
\begin{itemize}
    \item The components of $v$ are distinct.
    \item For all $j$, we have $k_j \leq g - 1$ and $k_j$ is even.
    \item The associated parity of spin structure is $g + 1 \hmod 2$.
    \item The image of $\wt{\cC}$ in $\Om\cM_g(\kap)$ is a nonhyperelliptic component.
\end{itemize}
By Theorem \ref{thm:KZ}, these assumptions imply $g \geq 4$. Consider a holomorphic $1$-form $(X_0,\om_0) \in \wt{\Om}\cM_g(\kap)$ from Case 1 as in (\ref{eq:X0w0}), and keep notation from Case 1. Suppose additionally that
\be
\int_{\al_1} \om_0 = \int_{\al_2} \om_0 .
\ee
Cut $X_0$ along $\al_1 \cup \al_2$, glue the top side of $\al_1$ to the bottom side of $\al_2$ to obtain a closed geodesic $\al_1^\pr$, and glue the top side of $\al_2$ to the bottom side of $\al_1$ to obtain a closed geodesic $\al_2^\pr$. Let
\be
(X_1,\om_1) \in \wt{\Om}\cM_g(\kap)
\ee
be the resulting holomorphic $1$-form. For $3 \leq j \leq g + n - 1$, the horizontal cylinder $C_j$ and the closed geodesic $\al_j$ are preserved. For $1 \leq j \leq 2$, let $C_j^\pr \subset X_1$ be the horizontal cylinder containing $\al_j^\pr$. The horizontal foliation of $(X_1,\om_1)$ is periodic, and the horizontal cylinders are $C_1^\pr,C_2^\pr,C_3,\dots,C_{g+n-1}$. For $3 \leq j \leq g$, the vertical cylinder $D_j$ and the closed geodesic $\bet_j$ are preserved. Since $k_{r^\pr} \geq 2$, for $r^\pr + 1 \leq j \leq n$, the vertical cylinder $D_{g+j-1}$ and the closed geodesic $\bet_{g+j-1}$ are also preserved. Let $T_1^\pr \subset X_1$ be the region bounded by $\al_1^\pr$ from above, by $\al_2^\pr$ from below, and by the slits from $T_1 \subset X_0$. Let $T_2^\pr \subset X_1$ be the region bounded by $\al_2^\pr$ from above, by $\al_1^\pr$ from below, and by the slits from $T_2 \subset X_0$. Let $\bet_1^\pr \subset X_1$ be a smooth oriented closed loop going upward from the top side of $\al_1^\pr$ in $T_2^\pr$, then crossing the slit bordering $T_1^\pr$ and $T_2^\pr$ from right to left, then rotating counterclockwise around the slits in $T_1^\pr$, and then going upward to the bottom side of $\al_1^\pr$ and closing up. Let $\bet_2^\pr \subset X_1$ be a smooth oriented closed loop disjoint from $\bet_1^\pr$ going upward from the top side of $\al_2^\pr$, then crossing the slit bordering $T_1^\pr$ and $T_2^\pr$ from left to right, and then going upward to the bottom side of $\al_2^\pr$ to close up. Figure \ref{fig:spin} shows the configuration of $\al_1^\pr,\bet_1^\pr,\al_2^\pr,\bet_2^\pr$ in $(X_1,\om_1)$.

\begin{figure}
    \centering
    \includegraphics[width=0.7\textwidth]{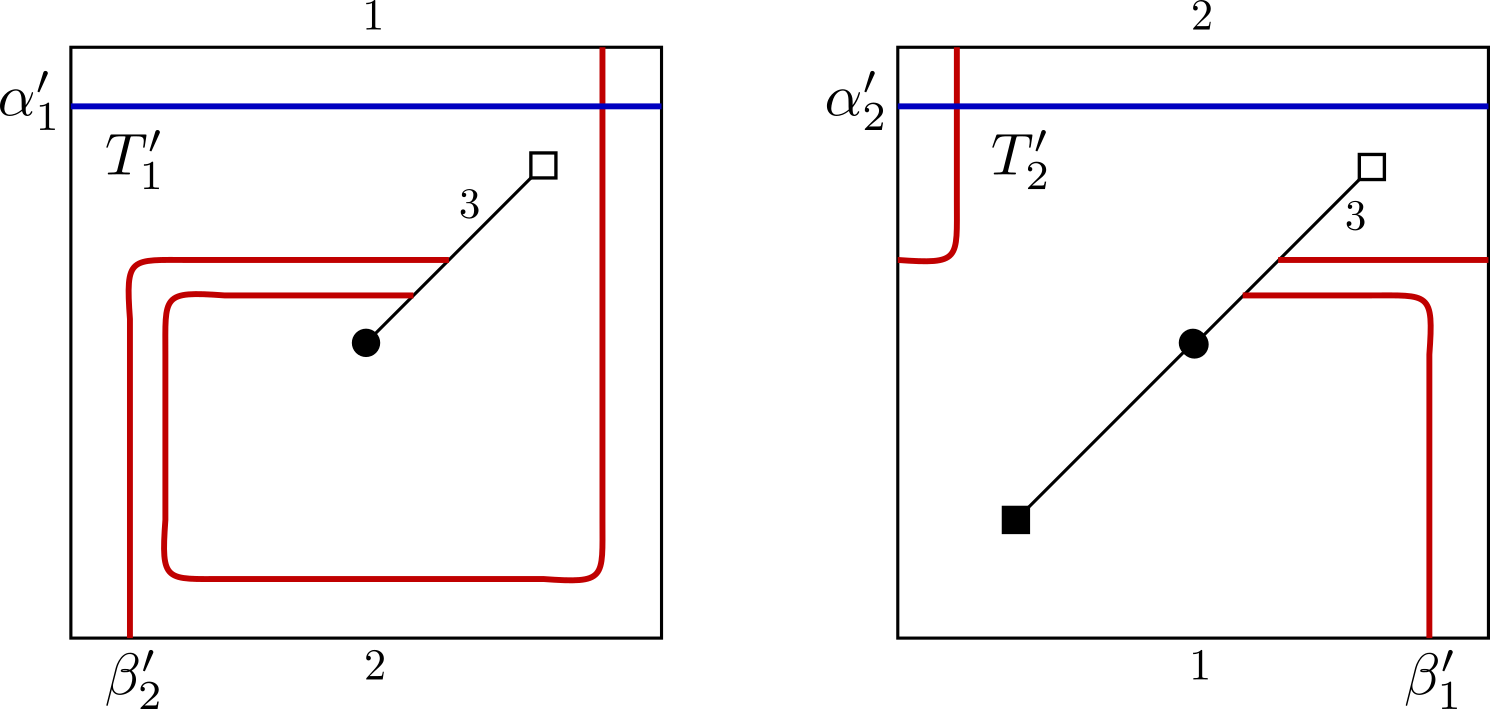}
    \caption{The configuration of the loops $\al_1^\pr,\bet_1^\pr,\al_2^\pr,\bet_2^\pr$ in $(X_1,\om_1)$ in Case 2.}
    \label{fig:spin}
\end{figure}

For $1 \leq j \leq 2$, let
\be
a_j^\pr = [\al_j^\pr] \in H_1(X_1;\Z), \quad b_j^\pr = [\bet_j^\pr] \in H_1(X_1;\Z) .
\ee
For $3 \leq j \leq g$, we have that $\bet_1^\pr \cup \bet_2^\pr$ is disjoint from $\al_j \cup \bet_j$. Moreover, $\bet_1^\pr$ is disjoint from $\al_2^\pr$ and intersects $\al_1^\pr$ once with $a_1^\pr \cdot b_1^\pr = 1$, and $\bet_2^\pr$ is disjoint from $\al_1^\pr$ and intersects $\al_2^\pr$ once with $a_2^\pr \cdot b_2^\pr = 1$. Therefore, 
\be
\{a_1^\pr,b_1^\pr,a_2^\pr,b_2^\pr\} \cup \{a_j,b_j\}_{j=3}^g
\ee
is a symplectic basis for $H_1(X_1;\Z)$. We have
\be
\ind(\al_1^\pr) = 0, \quad \ind(\al_2^\pr) = 0, \quad \ind(\bet_1^\pr) = 1, \quad \ind(\bet_2^\pr) = 0,
\ee
and for $3 \leq j \leq g$ we have $\ind(\al_j) = \ind(\bet_j) = 0$. Therefore, the parity of the associated spin structure is given by
\be
\phi(\om_1) = g + 1 \hmod 2 .
\ee
Since $k_1 \geq 2$, for $1 \leq j \leq r$, we have
\be
a_{g+j} = a_1^\pr + a_2^\pr + a_3 + \cdots + a_{k_1 + \cdots + k_j + 1}
\ee
and since $k_{r^\pr} \geq 2$, for $r^\pr + 1 \leq j \leq n$, we have
\be
a_{g+j-1} = a_{g - (k_j + \cdots + k_n)} + \cdots + a_g .
\ee
If $r^\pr = r + 1$, then $a_{g + r^\pr - 1}$ has already been described. If $r^\pr = r + 2$, then
\be
a_{g + r^\pr - 1} = \begin{cases} a_{g-(k_{r^\pr} + \cdots + k_g)} + \cdots + a_g, \quad \text{ if } k_{r^\pr} + \cdots + k_g < g - 2 , \\ a_1^\pr + a_3 + \cdots + a_g, \quad\quad\quad \; \, \text{ if } k_{r^\pr} + \cdots + k_g = g - 2 . \end{cases}
\ee

We have verified that $(X_1,\om_1)$ satisfies condition (3). For $1 \leq j \leq 2$, the bottom boundary of $C_j^\pr$ consists of a single saddle connection joining $Z_1$ to itself. The top boundary of $C_j^\pr$ also consists of a single saddle connection. If $i_{3-j}^\pr \neq 0$, this saddle connection joins $Z_{i_j^\pr}$ to itself, and otherwise, this saddle connection joins $Z_{r+1}$ to itself. The zeros in the top and bottom boundaries of $C_j^\pr$ are distinct. Thus, since $(X_0,\om_0)$ satisfies conditions (1) and (2), $(X_1,\om_1)$ also satisfies conditions (1) and (2).

Next, we describe the vertical cylinders on $(X_1,\om_1)$. There is a vertical cylinder $D_1^\pr \subset X_1$ passing through $T_1^\pr$ and $T_2^\pr$ and crossing the slit bordering $T_1^\pr$ and $T_2^\pr$. The left boundary of $D_1^\pr$ consists of a single saddle connection joining $Z_{r+1}$ to itself, and the right boundary of $D_1^\pr$ consists of a single saddle connection joining $Z_1$ to itself. There is a vertical cylinder $D_2^\pr \subset X_1$ passing through $T_1^\pr$ and $T_2^\pr$ and disjoint from the slits. The left boundary of $D_2^\pr$ consists of a single saddle connection joining $Z_1$ to itself. The right boundary of $D_2^\pr$ also consists of a single saddle connection. If $r^\pr = r + 1$ or $k_{r^\pr} + \cdots + k_g < g-2$, this saddle connection joins $Z_{r+1}$ to itself, and otherwise, this saddle connection joins $Z_{r^\pr}$ to itself. Thus, since $(X_0,-i\om_0)$ satisfies conditions (1) and (2), $(X_1,-i\om_1)$ also satisfies conditions (1) and (2).

For $1 \leq j \leq 2$, let $\bet_j^{\pr\pr} \subset D_j^\pr$
be a closed geodesic with $\int_{\bet_j^{\pr\pr}} \om_1 \in i\R_{>0}$. Letting $b_j^{\pr\pr} = [\bet_j^{\pr\pr}]$, we have
\be
b_1^{\pr\pr} = a_1^\pr - a_2^\pr + b_1^\pr, \quad b_2^{\pr\pr} = -a_2^\pr + b_1^\pr + b_2^\pr .
\ee
For $1 \leq j \leq r$, there is a vertical cylinder $D_{g+j}^\pr \subset X_1$ passing through $\bet_1^\pr$ and $\al_2^\pr$, and passing through $T_k$ if and only if $3 \leq k \leq k_1+\cdots+k_j+1$. Let $\bet_{g+j}^\pr \subset D_{g+j}^\pr$ be a closed geodesic with $\int_{\bet_{g+j}^\pr} \om_1 \in i\R_{>0}$. Let $b_{g+j}^\pr = [\bet_{g+j}^\pr]$. Since $k_1 \geq 2$, we have
\be
b_{g+j}^\pr = -a_1^\pr + b_2^\pr + b_3 + \cdots + b_{k_1+\cdots+k_j+1} .
\ee
Since $k_{r^\pr} \geq 2$, for $r^\pr + 1 \leq j \leq n$, we have
\be
b_{g+j-1} = b_{g - (k_j + \cdots + k_n)} + \cdots + b_g .
\ee
If $r^\pr = r + 1$, then we have addressed all of the vertical cylinders. If $r^\pr = r + 2$ and $k_{r^\pr} + \cdots + k_g < g - 2$, then the vertical cylinder $D_{g+r^\pr-1}$ and the closed geodesic $\beta_{g+r^\pr-1}$ are preserved, and we have
\be
b_{g+r+1} = b_{g - (k_{r^\pr} + \cdots + k_n)} + \cdots + b_g .
\ee
If $r^\pr = r + 2$ and $k_{r^\pr} + \cdots + k_g = g - 2$, then there is a vertical cylinder $D_{g+r^\pr-1}^\pr \subset X_1$ passing through $\al_1^\pr,\al_2^\pr,\bet_2^\pr,T_3,\dots,T_g$. Let $\bet_{g+r^\pr-1}^\pr \subset D_{g+r^\pr-1}^\pr$ be a closed geodesic with $\int_{\bet_{g+r^\pr-1}^\pr} \om_1 \in i\R_{>0}$. Letting $b_{g+r^\pr-1}^\pr = [\bet_{g+r^\pr-1}^\pr]$, we have
\be
b_{g+r^\pr-1}^\pr = -a_2^\pr + b_1^\pr + b_2^\pr + b_3 + \cdots + b_g .
\ee
We have verified that $(X_1,-i\om_1)$ satisfies condition (3).

Lastly, when $g \geq 5$ is odd and $\kap = (g-1,g-1)$, a similar argument from Case 1 considering $T_3,\dots,T_g$ shows that $(X_1,\om_1)$ does not admit a hyperelliptic involution.

Letting $(X,\om) = (X_1,\om_1)$, we are done with Case 2. \\

\paragraph{\bf Case 3:} We assume that the image of $\wt{\cC}$ in $\Om\cM_g(\kap)$ is a hyperelliptic component, which implies the following.
\begin{itemize}
    \item We have $\kap = (g-1,g-1)$.
    \item The components of $v$ are distinct.
\end{itemize}
Consider a holomorphic $1$-form $(X_0,\om_0) \in \wt{\Om}\cM_g(g-1,g-1)$ from Case 1 as in (\ref{eq:X0w0}), and keep notation from Case 1. Suppose additionally that for $1 \leq j \leq g$, we have
\be
\int_{\al_j} \om_0 = \int_{\al_{g+1-j}} \om_0 .
\ee
Cut $(X_0,\omega_0)$ along $\al_1 \cup \cdots \cup \al_g$, reglue opposite sides of $\al_j$ and $\al_{g+1-j}$, and let
\be
(X_2,\om_2) \in \wt{\Om}\cM_g(g-1,g-1)
\ee
be the resulting holomorphic $1$-form. There is a holomorphic involution $\tau : X_2 \ra X_2$ with $\tau^\ast \om_2 = -\om_2$ that preserves each horizontal cylinder and exchanges the two zeros of $\om_2$. Each horizontal cylinder contains two fixed points of $\tau$, and there are $g+1$ horizontal cylinders, so $\tau$ has $2g+2$ fixed points and must be a hyperelliptic involution.

For $1 \leq j \leq g$, let $\al_j^\pr$ be the closed geodesic on $(X_2,\om_2)$ arising from the top side of $\al_j$, and let $C_j^\pr$ be the horizontal cylinder containing $\al_j^\pr$. The horizontal cylinder $C_{g+1}$ and the closed geodesic $\al_{g+1}$ are preserved. The horizontal foliation of $(X_2,\om_2)$ is periodic, and the horizontal cylinders on $(X_2,\om_2)$ are $C_1^\pr,\dots,C_g^\pr,C_{g+1}$. Letting $a_j^\pr = [\al_j^\pr]$, $1 \leq j \leq g$, and $a_{g+1} = [\al_{g+1}]$, we have
\be
a_{g+1} = a_1^\pr + \cdots + a_g^\pr .
\ee
For $1 \leq j \leq g$, there is a loop $\ell_j$ which crosses $C_j^\pr$ and $C_{g+1}$ from bottom to top and is disjoint from the other horizontal cylinders. Thus, $(X_2,\om_2)$ satisfies condition (3). For $1 \leq j \leq g$, the top boundary of $C_j^\pr$ consists of a single saddle connection joining $Z_2$ to itself, and the bottom boundary of $C_j^\pr$ consists of a single saddle connection joining $Z_1$ to itself. The top boundary of $C_{g+1}$ consists of $g$ saddle connections joining $Z_1$ to itself, and the bottom boundary of $C_{g+1}$ consists of $g$ saddle connections joining $Z_2$ to itself. Thus, $(X_2,\om_2)$ satisfies conditions (1) and (2).

Next, we describe the vertical cylinders on $(X_2,\om_2)$. For $1 \leq j \leq g$, let $T_j^\pr \subset X_2$ be the region bounded by $\al_j^\pr$ from below, by $\al_{g+1-j}^\pr$ from above, and by the slits coming from $T_j$. For $1 \leq j \leq \lfloor\frac{g}{2}\rfloor$, there is a vertical cylinder $D_j^\pr$ passing through $T_j^\pr$ and $T_{g+1-j}^\pr$ and disjoint from the slits. There is a vertical cylinder $D_{\lfloor\frac{g}{2}\rfloor+1}^\pr$ that passes through $T_{\lfloor\frac{g}{2}\rfloor+1}^\pr$ and is disjoint from the slits if $g$ is odd, and that passes through $T_{\frac{g}{2}}^\pr,T_{\frac{g}{2}+1}^\pr$ and the slit bordering $T_{\frac{g}{2}}^\pr,T_{\frac{g}{2}+1}^\pr$ if $g$ is even. For $1 \leq j \leq \lfloor\frac{g-1}{2}\rfloor$, there is a vertical cylinder $D_{g+1-j}^\pr$ passing through $T_j^\pr,T_{j+1}^\pr,T_{g-j}^\pr,T_{g+1-j}^\pr$, as well as the slit bordering $T_j^\pr,T_{j+1}^\pr$ and the slit bordering $T_{g-j}^\pr,T_{g+1-j}^\pr$. There is a vertical cylinder $D_{g+1}^\pr$ passing through $T_1^\pr,T_g^\pr$, and the slit bordering $T_1^\pr,T_g^\pr$. The vertical foliation of $(X_2,\om_2)$ is periodic, and the vertical cylinders are $D_1^\pr,\dots,D_{g+1}^\pr$. For $1 \leq j \leq g+1$, let $\bet_j^\pr \subset D_j^\pr$ be a closed geodesic with $\int_{\bet_j^\pr} \om_2 \in i\R_{>0}$, and let $b_j^\pr = [\bet_j^\pr]$. The nonzero algebraic intersection numbers $a_j^\pr \cdot b_k^\pr$ are given by
\begin{align*}
a_j^\pr \cdot b_j^\pr = a_{g+1-j}^\pr \cdot b_j^\pr &= 1, \quad 1 \leq j \leq \lfloor\frac{g}{2}\rfloor, \\
a_{\lfloor\frac{g}{2}\rfloor+1}^\pr \cdot b_{\lfloor\frac{g}{2}\rfloor+1}^\pr &= 1, \\
a_{j+1}^\pr \cdot b_{g+1-j}^\pr = a_{g+1-j}^\pr \cdot b_{g+1-j}^\pr &= 1, \quad 1 \leq j \leq \lfloor\frac{g-1}{2}\rfloor, \\
a_1^\pr \cdot b_{g+1}^\pr &= 1 .
\end{align*}
We have verified that $(X_2,-i\om_2)$ satisfies condition (3). Moreover, we note that the span of $b_1^\pr,\dots,b_g^\pr$ in $H_1(X_2;\Z)$ contains homology classes $b_1^{\pr\pr},\dots,b_g^{\pr\pr}$ such that $\{a_j^\pr,b_j^{\pr\pr}\}_{j=1}^g$ is a symplectic basis for $H_1(X_2;\Z)$. For $1 \leq j \leq \lfloor\frac{g}{2}\rfloor$, the left boundary of $D_j^\pr$ consists of a single saddle connection joining $Z_1$ to itself, and the right boundary of $D_j^\pr$ consists of a single saddle connection joining $Z_2$ to itself. For $\lfloor\frac{g}{2}\rfloor+1 \leq j \leq g+1$, the left boundary of $D_j^\pr$ consists of a single saddle connection joining $Z_2$ to itself, and the right boundary of $D_j^\pr$ consists of a single saddle connection joining $Z_1$ to itself. Thus, $(X_2,-i\om_2)$ satisfies conditions (1) and (2).

Letting $(X,\om) = (X_2,\om_2)$, we are done with Case 3. See Figure \ref{fig:slittorihyp} for an example in the case $\kap = (2,2)$. \\

\begin{figure}
    \centering
    \includegraphics[width=0.5\textwidth]{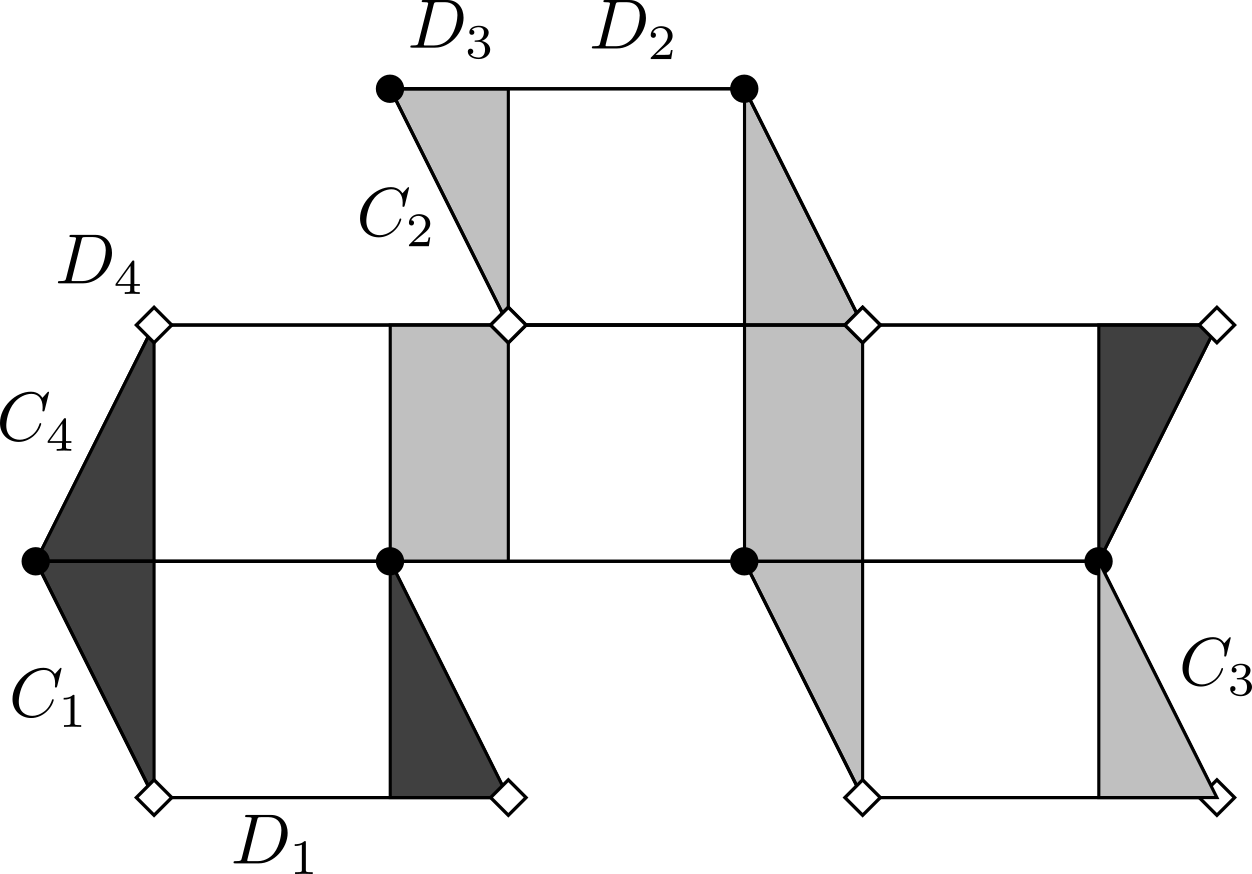}
    \caption{An example of $(X,\om)$ in the hyperelliptic component of $\Om\cM_3(2,2)$ arising from the construction in Case 3.}
    \label{fig:slittorihyp}
\end{figure}

\paragraph{\bf Case 4:} We make the following assumptions.
\begin{itemize}
    \item The components of $v$ are distinct.
    \item For some $j$, we have $k_j > g - 1$.
\end{itemize}
By relabelling zeros, we may assume that $k_n > g - 1$. Let $k_n^\pr = \sum_{j=1}^{n-1} k_j$. Then $\kap^\pr = (k_1,\dots,k_{n-1},k_n^\pr)$ is a partition of $2(g-a) - 2$ for some $0 < a < g - 1$, and we have $k_n = k_n^\pr + 2a$. By Cases 1, 2, and 3, there is $(Y,\eta) \in \wt{\Om}\cM_{g-a}(\kap^\pr)$ satisfying conditions (1), (2), and (3) in Theorem \ref{thm:constr}. Additionally, if the orders of the zeros are all even, we may assume that $\phi(\eta) = 0$ or $1$ according to whether the image of $\wt{\cC}$ in $\Om\cM_g(\kap)$ is an even component or an odd component. We apply the following surgery to $(Y,\eta)$.

Let $\gam$ be a horizontal segment starting at $Z_n$ and shorter than any saddle connection on $(Y,\eta)$. Slit along $\gam$, and let $\gam^+,\gam^-$ be the top and bottom sides of the slit, respectively. Divide $\gam^+$ into $2$ segments, with the left segment having length $\ell_1$ and the right segment having length $\ell_2$. Divide $\gam^-$ into $2$ segments, with the left segment having length $\ell_2$ and the right segment having length $\ell_1$. Identify pairs of segments of equal length, and let $(Y_1,\eta_1)$ be the resulting holomorphic $1$-form. On $(Y_1,\eta_1)$, the order of $Z_n$ is $k_n^\pr + 2$. The orders of the other zeros are unchanged, and no new zeros are created. No new horizontal cylinders are created, and the only horizontal saddle connections that are created are loops through $Z_n$. Since $(Y,\eta)$ satisfies conditions (1), (2), and (3), $(Y^\pr,\eta^\pr)$ also satisfies conditions (1), (2), and (3). By iterating this surgery $a$ times, we obtain a holomorphic $1$-form
\be
(X,\om) \in \wt{\Om}\cM_g(\kap)
\ee
satisfying conditions (1), (2), and (3) in Theorem \ref{thm:constr}. If some zero order is odd, then we are done, so suppose that all of the orders of the zeros are even. The above surgery can be alternatively described by slitting along $\gam$, gluing in a horizontal cylinder that does not contain a vertical saddle connection, and then vertically shrinking the cylinder to height $0$. The first step is an instance of bubbling a handle from Section 4 of \cite{KZ:components}. Since one of the endpoints of $\gam$ is not a zero, by Lemma 11 in \cite{KZ:components}, the above surgery does not change the parity of the associated spin structure. Thus, we are done with Case 4. \\

\paragraph{\bf Case 5:} We assume that the components of $v$ are not distinct.

The components of $v$ determine a partition $P_v = \{P_1,\dots,P_{n^\pr}\}$ of the set $\{1,\dots,n\}$, where $i$ and $j$ are in the same part of $P_v$ if and only if $v_i - v_j = 0$. Letting $\ell_j = \sum_{r \in P_j} k_r$, we obtain a partition $\lam = (\ell_1,\dots,\ell_{n^\pr})$ of $2g-2$ with $n^\pr > 1$ that is refined by $\kap$. Choose a representative $w \in \R^n$ of $v$, and let $v^\pr \in \R^{n^\pr}_0$ be represented by the vector $w^\pr \in \R^{n^\pr}$ such that for $1 \leq j \leq n^\pr$ and $r \in P_j$, $w^\pr_j = w_r$. Note that the components of $v^\pr$ are distinct. By Cases 1, 2, 3, and 4, there is $(Y,\eta) \in \wt{\Om}\cM_g(\lam)$ with zeros $Z_1^\pr,\dots,Z_{n^\pr}^\pr$ of orders $\ell_1,\dots,\ell_{n^\pr}$ and satisfying conditions (1), (2), and (3) of Theorem \ref{thm:constr}. If $k_j$ is even for all $j$, then $\ell_j$ is even for all $j$, and we may additionally assume that $\phi(\eta) = 0$ or $1$ according to whether the image of $\wt{\cC}$ in $\Om\cM_g(\kap)$ is an even component or an odd component. We apply the following surgery to $(Y,\eta)$.

For each $1 \leq j \leq n^\pr$, write $P_j = \{k_{j,1},\dots,k_{j,n_j}\}$. Apply the slit construction from Section \ref{sec:background} iteratively, using short horizontal segments on $(Y,\eta)$, in order to split the zero $Z_j^\pr$ of order $\ell_j$ into $n_j$ zeros of orders $k_{j,1},\dots,k_{j,n_j}$. After relabelling the resulting zeros, we obtain a holomorphic $1$-form $(X,\om) \in \wt{\Om}\cM_g(\kap)$. In each use of the slit construction above, we only modify a holomorphic $1$-form on a contractible neighborhood of a zero. Thus, in the case where every $k_j$ is even, we do not change the parity of the associated spin structure. Therefore, $(X,\om) \in \wt{\cC}$.

By construction, since $(Y,\eta)$ satisfies conditions (1) and (2), $(X,\om)$ also satisfies conditions (1) and (2). No new horizontal cylinders are created during the surgery, so since $(Y,\eta)$ satisfies condition (3), $(X,\om)$ also satisfies condition (3). We are done with Case 5.
\end{proof}

The criteria in Theorems \ref{thm:densecriterion} and \ref{thm:percriterion} can be readily verified for many other $\SL(2,\R)$-orbit closures in strata. As an example, we briefly discuss the existence of dense orbits of real Rel flows on nonarithmetic eigenform loci in $\wt{\Om}_1\cM_2(1,1)$. See \cite{BSW:horocycle} and \cite{McM:SL2R} for definitions and detailed discussions. Let $D > 0$ be a non-square integer such that $D \equiv 0, 1 \hmod 4$. Let
\be
\cE_D \subset \wt{\Om}_1\cM_2(1,1)
\ee
be the area-$1$ locus of eigenforms for real multiplication by the real quadratic order of discriminant $D$. There is $(X,\om) \in \cE_D$ with a periodic horizontal foliation and $3$ horizontal cylinders, whose circumferences have irrational ratios. Each horizontal saddle connection on $(X,\om)$ is a loop. Letting $v = (1,-1) \in \R^2_0$, $\Rel_{\R v}(X,\om)$ is well-defined, and its closure is given by
\be
\{u_s {\rm Rel}_{tv} (X,\om) \; : \; s, t \in \R \} .
\ee
By replacing $(X,\om)$ with $\Rel_{t v}(X,\om)$ for some $t \in \R$, we can ensure that the $\SL(2,\R)$-orbit of $(X,\om)$ is dense in $\cE_D$. Theorem \ref{thm:densecriterion} then implies the following.

\begin{thm} \label{thm:eigenform}
For any non-square integer $D > 0$ such that $D \equiv 0, 1 \hmod 4$, there exists $(Y,\eta) \in \cE_D$ such that $\Rel_{\R v}(Y,\eta)$ is dense in $\cE_D$.
\end{thm}

Lastly, we use the constructions in our proof of Theorem \ref{thm:constr} to prove Theorem \ref{thm:denseV}. We refer to Figure \ref{fig:slittori11} and the discussion preceding the proof of Theorem \ref{thm:constr} for the simplest case where $\kap = (1,1)$.

\begin{thm} \label{thm:perdense}
Suppose $n > 1$, fix a nonzero $v \in \R^n_0$ with distinct components, and let $V = \C v$. Let $\wt{\cC}_1$ be a connected component of $\wt{\Om}_1\cM_g(\kap)$. There is an explicit $(X,\om) \in \wt{\cC}_1$ with periodic horizontal and vertical foliations, such that the leaf of $\cA_V(\kap)$ through $(X,\om)$ is dense in $\wt{\cC}_1$.
\end{thm}

\begin{proof}
First, suppose there is $(X,\om) \in \wt{\cC}_1$ arising from the construction in Case 1 of the proof of Theorem \ref{thm:constr}, and keep notation from Case 1. Recall that $(X,\om)$ has a periodic horizontal foliation with $g + n - 1$ horizontal cylinders $C_1,\dots,C_{g+n-1}$, and a periodic vertical foliation with $g + n - 1$ vertical cylinders $D_1,\dots,D_{g+n-1}$. The circumference of $C_j$ is $w_j$, and the circumference of $D_j$ is $w_j^\pr$. The circumferences $w_1,\dots,w_g$ and $w_1^\pr,\dots,w_g^\pr$ are parameters in the construction in Case 1, arising from the choice of tori $T_1,\dots,T_g$ with periodic horizontal and vertical foliations. This construction also depended on parameters $z_1,\dots,z_n \in \C$ for the slits joining the tori. Moreover, $w_1,\dots,w_g$, $iw_1^\pr,\dots,iw_g^\pr$, and $z_1-z_{r+1},\dots,z_r-z_{r+1},z_{r+2}-z_{r+1},\dots,z_n-z_{r+1}$ provide period coordinates for $(X,\om)$. The valid choices of parameters yield period coordinates ranging over an open subset of $\R^g \times (i\R)^g \times \C^{n-1}$. As in Section \ref{sec:twist}, choose functions
\be
t,b,\ell,r : \{1,\dots,g+n-1\} \ra \{1,\dots,n\},
\ee
so that $t(j)$ (respectively, $b(j)$) is the index of a zero in the top (respectively, bottom) boundary of $C_j$, and $\ell(j)$ (respectively, $r(j)$) is the index of a zero in the left (respectively, right) boundary of $D_j$. Both $(X,\om)$ and $(X,-i\om)$ satisfy the conditions in Theorem \ref{thm:constr}. Then by the proofs of Lemmas \ref{lem:QRindep} and \ref{lem:recip}, we can choose the parameters $w_1,\dots,w_g$, $w_1^\pr,\dots,w_g^\pr$, and $z_1,\dots,z_n$ so that the following four properties hold.
\begin{itemize}
    \item $(X,\om)$ has area $1$.
    \item $(v_{t(j)} - v_{b(j)})/w_j$, $1 \leq j \leq g + n - 1$, are linearly independent over $\Q$.
    \item $(v_{\ell(j)} - v_{r(j)})/w_j^\pr$, $1 \leq j \leq g + n - 1$, are linearly independent over $\Q$.
    \item The period coordinates of $(X,\om)$ are linearly independent over $\ol{\Q} \cap \R$.
\end{itemize}
Let $L$ be the leaf of $\cA_V(\kap)$ through $(X,\om)$, and recall that the $\SL(2,\R)$-action on $\wt{\cC}_1$ sends leaves of $\cA_V(\kap)$ to leaves of $\cA_V(\kap)$. By Corollary \ref{cor:realRelhoro}, the closure of ${\rm Rel}_{\R v}(X,\om)$ contains the horocycle $u_\R (X,\om)$. Similarly, by Corollary \ref{cor:realRelhoro} applied to $(X,-i\om)$, the closure of ${\rm Rel}_{i\R v}(X,\om)$ contains the opposite horocycle $v_\R (X,\om)$. Since $u_\R$ and $v_\R$ generate $\SL(2,\R)$, it follows that the closure $\ol{L}$ is $\SL(2,\R)$-invariant. By Theorem \ref{thm:Wdense}, the $\SL(2,\R)$-orbit of $(X,\om)$ is dense in $\wt{\cC}_1$. Thus, $L$ is dense in $\wt{\cC}_1$.

The argument when there is $(X,\om) \in \wt{\cC}_1$ arising from Case 2 of the proof of Theorem \ref{thm:constr} is similar. In the notation from Case 2, the horizontal and vertical foliations of $(X,\om)$ are periodic, and there are $g + n - 1$ horizontal cylinders with circumferences $w_1,\dots,w_{g+n-1}$, and $g + n - 1$ vertical cylinders with circumferences $w_1^\pr,\dots,w_{g+n-1}^\pr$. The circumferences $w_3,\dots,w_g$ and $w_3^\pr,\dots,w_g^\pr$ are parameters in the construction, arising from the choice of tori $T_3,\dots,T_g$. There are parameters $z_1,\dots,z_n \in \C$ arising from the slits. The circumferences $w_1^\pr,w_2^\pr$ arise from cylinders passing through $T_1^\pr$ and $T_2^\pr$, and can be freely varied by changing the choices of horizontal closed geodesics in $T_1$ and $T_2$ to cut along. Lastly, in the beginning of the construction, we assumed $w_1 = w_2$. However, $w_1,\dots,w_g$, $iw_1^\pr,\dots,iw_g^\pr$, and $z_1-z_{r+1},\dots,z_r-z_{r+1},z_{r+2}-z_{r+1},\dots,z_n-z_{r+1}$ provide period coordinates for $(X,\om)$, since the associated elements of $H_1(X,Z(\om);\Z)$ are linearly independent, which is shown in Case 2. By perturbing these period coordinates within $\R^g \times (i\R)^g \times \C^{n-1}$, we can perturb $(X,\om)$ while preserving each of the horizontal and vertical cylinders and preserving the property of having periodic horizontal and vertical foliations, so that the four properties from the previous paragraph hold. Then the leaf of $\cA_V(\kap)$ through $(X,\om)$ is dense in $\wt{\cC}_1$. The argument when there is $(X,\om) \in \wt{\cC}_1$ arising from Case 3 of the proof of Theorem \ref{thm:constr} is also similar. The horizontal and vertical foliations of $(X,\om)$ are periodic, and there are $g + 1$ horizontal cylinders with circumferences $w_1,\dots,w_{g+1}$, and $g + 1$ vertical cylinders with circumferences $w_1^\pr,\dots,w_{g+1}^\pr$. There is a parameter $z_1 \in \C$ arising from the slits. As is shown in Case 3, $w_1,\dots,w_g$, $iw_1^\pr,\dots,iw_g^\pr$, and $z_1$ provide period coordinates for $(X,\om)$. Again, we can perturb these period coordinates within $\R^g \times (i\R)^g \times \C$ so that the four properties in the previous paragraph hold. Then the leaf of $\cA_V(\kap)$ through $(X,\om)$ is dense in $\wt{\cC}_1$.

\begin{figure}
    \centering
    \includegraphics[width=0.7\textwidth]{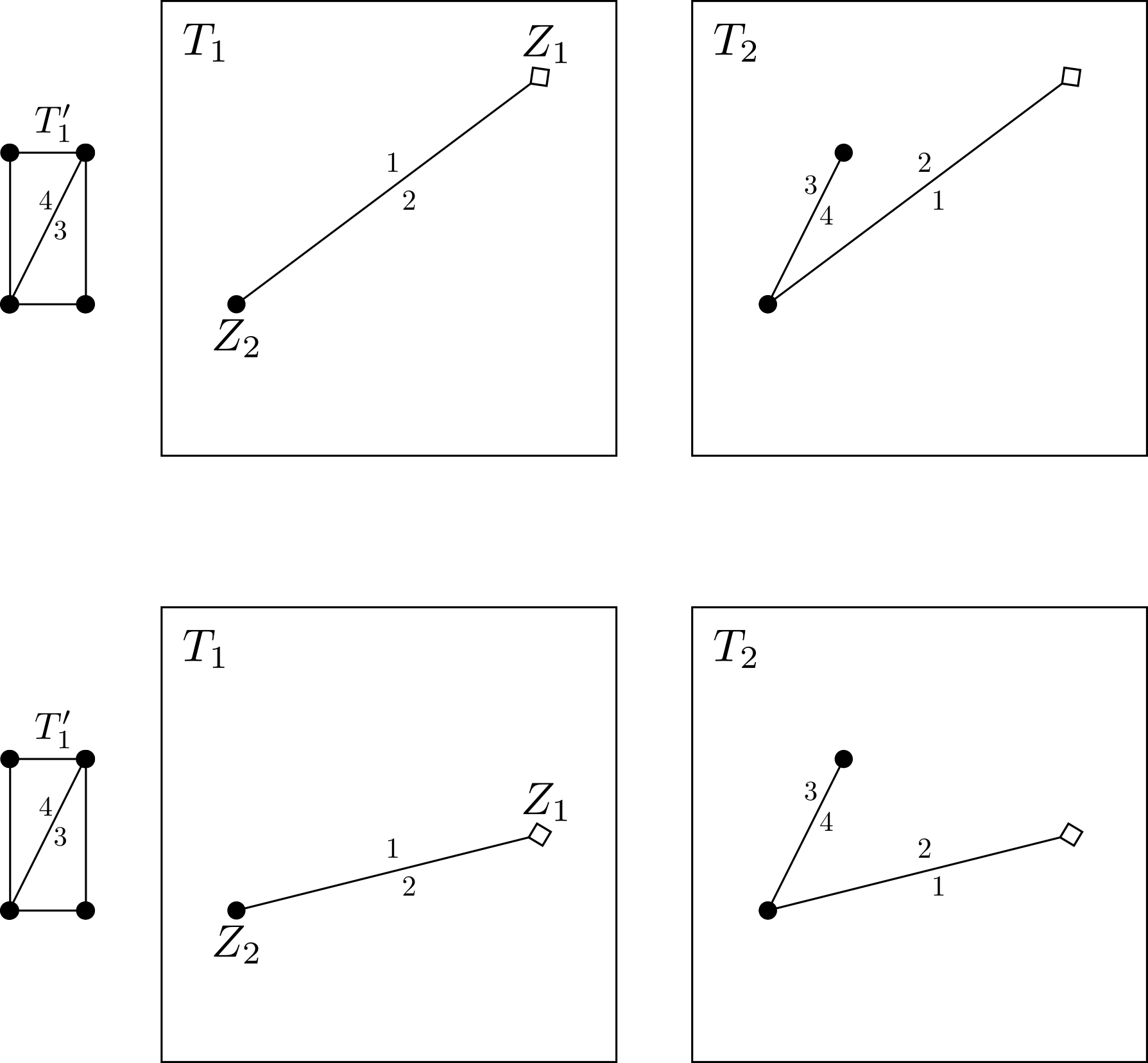}
    \caption{An example of $(X,\om) \in \wt{\Om}\cM_3(1,3)$ (top) and ${\rm Rel}_{-i\eps_1v}(X,\om)$ (bottom) arising from the construction in the proof of Theorem \ref{thm:perdense}.}
    \label{fig:slittori31}
\end{figure}

We may now assume that some $k_j > g - 1$. By relabelling zeros, we may assume that $k_n > g - 1$. Let $k_n^\pr = \sum_{j=1}^{n-1} k_j$. Then $\kap^\pr = (k_1,\dots,k_{n-1},k_n^\pr)$ is a partition of $2(g-a)-2$ for some $0 < a < g - 1$, and $k_n = k_n^\pr + 2a$. Choose $(Y,\eta) \in \wt{\Om}\cM_{g-a}(\kap^\pr)$ arising from the construction in Case 1, 2, or 3 of the proof of Theorem \ref{thm:constr}. If every $k_j$ is even, we additionally require that $\phi(\eta) = 0$ or $1$ according to whether the image of $\wt{\cC}_1$ in $\Om\cM_g(\kap)$ lies in an even component or odd component.

We focus on the case where $(Y,\eta)$ arises from the construction in Case 1, since the analysis in the other cases is similar. Keep notation from Case 1. Recall that $(Y,\eta)$ has $(g-a) + n - 1$ horizontal cylinders $C_1,\dots,C_{(g-a)+n-1}$ with heights $h_1,\dots,h_{(g-a)+n-1}$ and circumferences $w_1,\dots,w_{(g-a)+n-1}$, and that $(Y,\eta)$ has $(g-a) + n - 1$ vertical cylinders $D_1,\dots,D_{(g-a)+n-1}$ with heights $h_1^\pr,\dots,h_{(g-a)+n-1}^\pr$ and circumferences $w_1^\pr,\dots,w_{(g-a)+n-1}^\pr$. For $1 \leq j \leq (g-a) + n - 1$, recall that $v_{t(j)} - v_{b(j)} \neq 0$ and $v_{\ell(j)} - v_{r(j)} \neq 0$, and let
\be
\del_j = \frac{h_j}{|v_{t(j)} - v_{b(j)}|} > 0, \quad \del_j^\pr = \frac{h_j^\pr}{|v_{\ell(j)} - v_{r(j)}|} > 0,
\ee
and let
\be
\eps = \min_j \del_j, \quad \eps^\pr = \min_j \del_j^\pr .
\ee
For $-\eps < t < \eps$, as we travel along ${\rm Rel}_{i\R v}(Y,\eta)$ from $(Y,\eta)$ to ${\rm Rel}_{itv} (Y,\eta)$, the height of $C_j$ changes by $t(v_{t(j)} - v_{b(j)})$, and we have $|t(v_{t(j)} - v_{b(j)})| < h_j$. Therefore, all of the horizontal cylinders on $(Y,\eta)$ persist on ${\rm Rel}_{itv}(Y,\eta)$. Similarly, for $-\eps^\pr < t < \eps^\pr$, all of the vertical cylinders on $(Y,\eta)$ persist on ${\rm Rel}_{tv}(Y,\eta)$. We will apply a variant of the surgery from Case 4 to $(Y,\eta)$, after imposing some additional restrictions on the parameters in the construction of $(Y,\eta)$.

We will additionally assume that $\del_{(g-a)+n-1} < \del_j$ and $\del_{(g-a)+n-1}^\pr < \del_j^\pr$ for all $1 \leq j < (g-a) + n - 1$. Recall that $(Y,\eta)$ is constructed from $g-a$ flat tori $T_1,\dots,T_{g-a}$ by iteratively applying the slit construction from Section \ref{sec:background}. Since $k_n^\pr = \sum_{j=1}^{n-1} k_j = (g - a) - 1$, we have $r = n - 1$ and $r^\pr = n$. In this case, there are $n$ steps in the construction in Case 1. The torus $T_{g-a}$ is only slit in steps $n-1$ and $n$. In step $n-1$, $T_{g-a}$ is slit along the projection of the segment from $0$ to $z_{n-1}$, and in step $n$, $T_{g-a}$ is slit along the projection of the segment from $0$ to $z_n$. There is a unique horizontal cylinder $C_{(g-a)+n-1}$ passing through all $g-a$ tori, and a unique vertical cylinder $D_{(g-a)+n-1}$ passing through all $g-a$ tori. The only other horizontal cylinder passing through $T_{g-a}$ is $C_{g-a}$, and the only other vertical cylinder passing through $T_{g-a}$ is $D_{g-a}$. For $1 \leq j \leq a$, choose $s_j,t_j \in \R$ satisfying
\begin{align}
0 < s_a < \cdots < s_1 < -x_n \label{eq:smalls} \\
0 < t_a < \cdots < t_1 < -y_n \label{eq:smallt} \\
0 < \frac{y_n}{x_n} < \frac{t_1}{s_1} < \cdots < \frac{t_a}{s_a} \label{eq:slopest}
\end{align}
and let $T_j^\pr = (\C/(\Z s_j + i\Z t_j),dz)$. The flat torus $T_j^\pr$ contains a closed geodesic $\gam_j$ satisfying $\int_{\gam_j} dz = s_j + it_j$. Let $\gam_j^\pr$ be the straight segment on $T_{g-a} \subset Y$ emanating from $Z_n$ and satisfying $\int_{\gam_j^\pr} \eta = s_j + it_j$. Slit $Y$ along $\gam_j^\pr$, slit $T_j^\pr$ along $\gam_j$, and reglue opposite sides. The result is a holomorphic $1$-form $(X,\om) \in \wt{\Om}\cM_g(\kap)$. In the case where all $k_j$ are even, the associated parity of spin structure is unchanged by Lemma 11 in \cite{KZ:components}. Every horizontal cylinder and every vertical cylinder on $(Y,\eta)$ persists on $(X,\om)$. The height of $C_{(g-a)+n-1}$ is decreased, and the height of $D_{(g-a)+n-1}$ is decreased, while the heights of the other horizontal cylinders and vertical cylinders are unchanged. Additionally, for $1 \leq j \leq a$, $(X,\om)$ has a horizontal cylinder $C_{(g-a)+n-1+j}$ passing through $T_1,\dots,T_{g-a},T_1^\pr,\dots,T_j^\pr$ with circumference
\be
w_{(g-a)+n-1+j} = w_1 + \cdots + w_{g-a} + s_1 + \cdots + s_j
\ee
and a vertical cylinder $D_{(g-a)+n-1+j}$ passing through $T_1,\dots,T_{g-a},T_1^\pr,\dots,T_j^\pr$ with circumference
\be
w_{(g-a)+n-1+j}^\pr = w_1^\pr + \cdots + w_{g-a}^\pr + t_1 + \cdots + t_j .
\ee
We also denote $C_j^\pr = C_{(g-a)+n-1+j}$ and $D_j^\pr = D_{(g-a)+n-1+j}$.

The holomorphic $1$-forms $(X,\om)$ and $(X,-i\om)$ satisfy conditions (1) and (3) in Theorem \ref{thm:constr}. The cylinders $C_j^\pr$ and $D_j^\pr$ only contain the zero $Z_n$ in their boundary, and thus do not satisfy condition (2). However, $w_1,\dots,w_{g-a},s_1,\dots,s_a$, $iw_1^\pr,\dots,iw_{g-a}^\pr,it_1,\dots,it_a$, and $z_1-z_n,\dots,z_{n-1}-z_n$ provide period coordinates for $(X,\om)$. Again, by the proofs of Lemmas \ref{lem:realReltwist} and \ref{lem:QRindep}, we can perturb these period coordinates within $\R^g \times (i\R)^g \times \C^{n-1}$ so that the following properties hold.
\begin{itemize}
    \item $(X,\om)$ has area $1$.
    \item The collection
    \be
    \left\{ \frac{v_{t(j)} - v_{b(j)}}{w_j} \; : \; 1 \leq j \leq (g-a) + n - 1 \right\} \cup \left\{ \frac{v_{n-1} - v_n}{w_{(g-a)+n-1+j}}, \frac{v_{n-1} - v_n}{w_{g-a} + \sum_{k=1}^j s_k} \; : \; 1 \leq j \leq a \right\}
    \ee
    is linearly independent over $\Q$.
    \item The collection
    \be
    \left\{ \frac{v_{\ell(j)} - v_{r(j)}}{w_j^\pr} \; : \; 1 \leq j \leq (g-a) + n - 1 \right\} \cup \left\{ \frac{v_{n-1} - v_n}{w_{(g-a)+n-1+j}^\pr}, \frac{v_{n-1} - v_n}{w_{g-a}^\pr + \sum_{k=1}^j t_k} \; : \; 1 \leq j \leq a \right\}
    \ee
    is linearly independent over $\Q$.
    \item The period coordinates of $(X,\om)$ are linearly independent over $\ol{\Q} \cap \R$.
\end{itemize}

Replacing $v$ with $-v$ if necessary, we may assume that $v_{n-1} - v_n > 0$. Then there exist
\be
0 < \eps_1 < \cdots < \eps_a < \eps
\ee
with the following property. On ${\rm Rel}_{-itv}(X,\om)$, as $t$ increases from $0$ to $\eps_1$, the horizontal cylinder $C_{(g-a)+n-1}$ collapses to height $0$ and a new horizontal cylinder emerges, which we also denote $C_{(g-a)+n-1}$. No other horizontal cylinders collapse, although their heights may change. Let $(X_1,\om_1) = {\rm Rel}_{-i\eps_1v}(X,\om)$. On $(X_1,\om_1)$, the cylinder $C_{(g-a)+n-1}$ has circumference $w_{g-a} + s_1$. No other horizontal cylinder circumferences change. On $(X_1,\om_1)$, the cylinder $C_{(g-a)+n-1}$ contains $Z_n$ in its top boundary and $Z_{n-1}$ in its bottom boundary, the cylinder $C_1^\pr$ contains $Z_{n-1}$ in its top boundary and $Z_n$ in its bottom boundary, and the cylinder $C_{g-a}$ contains $Z_n$ in its top boundary and $Z_n$ in its bottom boundary. The zeros in the boundaries of all other horizontal cylinders remain the same. Similarly, for $2 \leq j \leq a$, on ${\rm Rel}_{-itv}(X,\om)$, as $t$ increases from $\eps_{j-1}$ to $\eps_j$, the horizontal cylinder $C_{j-1}^\pr$ collapses to height $0$ and a new horizontal cylinder emerges, which we also denote $C_{j-1}^\pr$. No other horizontal cylinders collapse. Let $(X_j,\om_j) = {\rm Rel}_{-i\eps_jv}(X,\om)$. On $(X_j,\om_j)$, the cylinder $C_{j-1}^\pr$ has circumference $w_{g-a} + s_1 + \cdots + s_j$. No other horizontal cylinder circumferences change. On $(X_j,\om_j)$, the cylinder $C_{j-1}^\pr$ contains $Z_n$ in its top boundary and $Z_{n-1}$ in its bottom boundary, the cylinder $C_j^\pr$ contains $Z_{n-1}$ in its top boundary and $Z_n$ in its bottom boundary, and the cylinder $C_{j-2}^\pr$ contains $Z_n$ in its top and bottom boundaries, where $C_0^\pr = C_{(g-a)+n-1}$. The zeros in the boundaries of all other horizontal cylinders remain the same. See Figure \ref{fig:slittori31} for an example in the case $\kap = (1,3)$.

The horizontal cylinders on $(X,\om)$ that are twisted along ${\rm Rel}_{\R v}(X,\om)$ are $C_1,\dots,C_{(g-a)+n-1}$. Since $(v_{t(j)} - v_{b(j)})/w_j$, $1 \leq j \leq (g-a) + n - 1$, are linearly independent over $\Q$,  by Lemma \ref{lem:realReltwist}, the closure of ${\rm Rel}_{\R v}(X,\om)$ contains all holomorphic $1$-forms obtained from $(X,\om)$ by twisting $C_1,\dots,C_{(g-a)+n-1}$ arbitrarily. The horizontal cylinders on $(X_1,\om_1)$ that are twisted along ${\rm Rel}_{\R v}(X_1,\om_1)$ are $C_1,\dots,C_{(g-a)-1}$, $C_{(g-a)+1},\dots,C_{(g-a)+n-1},C_1^\pr$. Since $(v_{n-1}-v_n)/w_{(g-a)+n}$, $(v_{t(j)}-v_{b(j)})/w_j$, $1 \leq j \leq (g-a) + n - 1$, are linearly independent over $\Q$, by Lemma \ref{lem:realReltwist}, the closure of ${\rm Rel}_{\R v}(X_1,\om_1)$ contains all holomorphic $1$-forms obtained from $(X_1,\om_1)$ by twisting $C_1,\dots,C_{(g-a)-1}$, $C_{(g-a)+1},\dots,C_{(g-a)+n-1},C_1^\pr$ arbitrarily. Similarly, for $2 \leq j \leq a$, the closure of ${\rm Rel}_{\R v}(X_j,\om_j)$ contains all holomorphic $1$-forms obtained from $(X_j,\om_j)$ by twisting the cylinders $C_1,\dots,C_{(g-a)-1}$, $C_{(g-a)+1},\dots,C_{(g-a)+n-2},C_{j-1}^\pr,C_j^\pr$ arbitrarily. Now let $L$ be the leaf of $\cA_V(\kap)$ through $(X,\om)$. For $1 \leq j \leq a$, by considering sequences of the form
\be
{\rm Rel}_{t_2 v} {\rm Rel}_{i\eps_jv} {\rm Rel}_{t_1 v} {\rm Rel}_{-i\eps_jv} {\rm Rel}_{t_0 v}(X,\om) \in L
\ee
with $t_0,t_1,t_2 \in \R$ such that ${\rm Rel}_{t_0 v}(X,\om)$ is close to $(X,\om)$ and such that ${\rm Rel}_{t_1 v} {\rm Rel}_{-i\eps_jv} {\rm Rel}_{t_0 v}(X,\om)$ is close to ${\rm Rel}_{-i\eps_jv} {\rm Rel}_{t_0 v}(X,\om)$, we see that the closure $\ol{L}$ contains all holomorphic $1$-forms obtained from $(X,\om)$ by twisting $C_1,\dots,C_{(g-a)+n-1},C_j^\pr$ arbitrarily. More generally, by considering appropriate sequences of the form
\be
{\rm Rel}_{t_{j+1}v} {\rm Rel}_{i\eps_j v} {\rm Rel}_{t_j v} {\rm Rel}_{i(\eps_{j-1}-\eps_j)v} \cdots {\rm Rel}_{t_2 v} {\rm Rel}_{i(\eps_1-\eps_2)v} {\rm Rel}_{t_1 v} {\rm Rel}_{-i\eps_1v} {\rm Rel}_{t_0 v}(X,\om) \in L
\ee
with $t_0,\dots,t_{j+1} \in \R$, we conclude that $\ol{L}$ contains the image of the horizontal twist map $\rho_\om$, and in particular contains the horocycle $u_\R (X,\om)$. An analogous argument, with $D_j$ and $D_j^\pr$ in place of $C_j$ and $C_j^\pr$, and swapping ${\rm Rel}_{iv}$ and ${\rm Rel}_v$, shows that $\ol{L}$ contains the opposite horocycle $v_\R (X,\om)$. Thus, $\ol{L}$ is $\SL(2,\R)$-invariant. By Theorem \ref{thm:Wdense}, the $\SL(2,\R)$-orbit of $(X,\om)$ is dense in $\wt{\cC}_1$. Thus, $L$ is dense in $\wt{\cC}_1$.

The remaining cases where $(Y,\eta)$ arises from the construction in Case 2 or Case 3 are similar. Again, $(Y,\eta)$ is constructed from $g-a$ flat tori $T_1,\dots,T_{g-a}$ by iteratively applying the slit construction. We choose $s_j,t_j \in \R$ satisfying (\ref{eq:smalls})-(\ref{eq:slopest}) and apply the same surgery as before with $T_{g-a}$ and $T_j^\pr = (\C/(\Z s_j + i\Z t_j),dz)$. The rest of the argument proceeds similarly.
\end{proof}


\bibliographystyle{math}
\bibliography{my.bib}

{\small
\noindent
Email: kwinsor@math.harvard.edu

\noindent
Department of Mathematics, Harvard University, Cambridge, Massachusetts, USA
}

\end{document}